\documentclass[a4paper,11pt]{amsart}

\usepackage[utf8]{inputenc}
\usepackage[english]{babel}
\usepackage{indentfirst}
\usepackage{csquotes}
\usepackage{soul} 

\sethlcolor{green} 

\usepackage{enumerate}

\usepackage[pagewise]{lineno}

\usepackage{amsmath,amsthm,amssymb,mathrsfs,thmtools}
\usepackage{dsfont}

\usepackage{enumitem}
\usepackage[toc,page]{appendix}

\usepackage{graphicx}
\usepackage{mathtools,array}
\usepackage[all]{xy}
\usepackage{tikz}
\usetikzlibrary{cd}
\usetikzlibrary{fit}
\usetikzlibrary{knots}
\usetikzlibrary{positioning}
\usetikzlibrary{arrows.meta}

\usepackage{fancyhdr,floatpag}
\usepackage{upgreek,color}
\usepackage{hyperref}
\usepackage[capitalise]{cleveref}

\usepackage{blkarray}
\usepackage{multirow}

\usepackage{dutchcal}

\usepackage[style=numeric,sorting=nyt,maxbibnames=7,maxcitenames=5,block=space,backend=bibtex]{biblatex}
\renewbibmacro{in:}{}

\hypersetup{linkbordercolor={0 0.60 0.75}, linkcolor=MidnightBlue}
\hypersetup{citebordercolor={0 0.60 0.75}, citecolor=MidnightBlue}

\numberwithin{equation}{section}
\theoremstyle{plain}

\newtheorem{thm}{Theorem}[section]
\newtheorem{conj}[thm]{Conjecture}

\newtheorem{dfn}[thm]{Definition}
\newtheorem{lem}[thm]{Lemma}

\newtheorem{pps}[thm]{Proposition}

\theoremstyle{definition}
\newtheorem{rmkx}[thm]{Remark}
\newenvironment{rmk}
  {\pushQED{\qed}\renewcommand{\qedsymbol}{$\triangle$}\rmkx}
  {\popQED\endrmkx}
\newtheorem{examplex}[thm]{Example}
\newenvironment{exm}
  {\pushQED{\qed}\renewcommand{\qedsymbol}{$\triangle$}\examplex}
  {\popQED\endexamplex}

\declaretheoremstyle[
  spaceabove=-6pt,
  spacebelow=6pt,
  headfont=\normalfont\bfseries,
  postheadspace=1em,
  qed=\qedsymbol,
  headpunct={}
]{mystyle} 
\declaretheorem[name={proofof},style=mystyle,unnumbered,
]{proofof}  
\makeatletter
\renewenvironment{proofof}[1] {\par\pushQED{\qed}\normalfont\topsep6\p@\@plus6\p@\relax\trivlist  \item[\hskip\labelsep
        \bfseries
    Proof of #1.]\ignorespaces}{\popQED\endtrivlist\@endpefalse}
\makeatother
\makeatletter
\newcommand{\thickhline}{
    \noalign {\ifnum 0=`}\fi \hrule height 1pt
    \futurelet \reserved@a \@xhline
}
\newcolumntype{"}{@{\hskip\tabcolsep\vrule width 1pt\hskip\tabcolsep}}
\makeatother

\makeatletter
\renewenvironment{proof}[1][\proofname] {\par\pushQED{\qed}\normalfont\topsep6\p@\@plus6\p@\relax\trivlist\item[\hskip\labelsep\bfseries#1\@addpunct{.}]\ignorespaces}{\popQED\endtrivlist\@endpefalse}
\makeatother

\makeatletter
\renewcommand{\@secnumfont}{\bfseries}
\makeatother
\usepackage{etoolbox}
\makeatletter
\patchcmd{\section}{\scshape}{\bf}{}{}
\patchcmd{\subsection}{\scshape}{\bf}{}{}
\patchcmd{\subsubsection}{\scshape}{\bf}{}{}
\makeatother

\newcommand{\N}{\mathbb{N}} 
\newcommand{\Z}{\mathbb{Z}} %integers
\newcommand{\Q}{\mathbb{Q}} %rationals
\newcommand{\R}{\mathbb{R}} %reals
\newcommand{\K}{\mathbb{K}} % Koerper!
\newcommand{\F}{\mathbb{F}} %finite field
\newcommand{\OS}{\mathcal{O}_S} %ring of S-integers 1

\newcommand{\anel}[1]{#1_{\text{ring}}} %subscript 'ring', e.g. for Aut_ring(R)
\newcommand{\ri}{\mathcal{O}} %ring of integers
\DeclareMathOperator{\GL}{GL} %general linear
\DeclareMathOperator{\SL}{SL} %special linear
\newcommand{\eij}{e_{i,j}} %elementary matrix in the group
\newcommand{\di}{d_i} %elementary diagonal matrix in the group
\newcommand{\ekl}[1]{e_{#1}} %elementary matrix in the group with index-inputs
\newcommand{\mbU}{\mathbf{U}} %unipotent (or upper unitriangular in GLn)
\newcommand{\mbD}{\mathbf{D}} %diagonal (or torus in GLn)
\DeclareMathOperator{\Diag}{Diag} %diagonal matrices
\newcommand{\sbgpdk}[1]{\mathcal{D}_{#1}} %subgroup generated by the d_i(u), with index-input

\newcommand{\Mult}{\mathbb{G}_{m}} %multiplicative group scheme
\newcommand{\Addi}{\mathbb{G}_{a}} %additive group scheme
\newcommand{\Aff}{\mathbb{A}\mathrm{ff}} %the affine group
\newcommand{\AffR}[1]{\mathbb{G}_{a}(#1) \rtimes \mathbb{G}_{m}(#1)} %affine group with input
\newcommand{\mbB}{\mathbf{B}} %Borel (or upper triangular in GLn)
\newcommand{\mbG}{\mathbf{G}} %'generic' group scheme
\newcommand{\mbW}{\mathbf{W}} %Z sd D
\newcommand{\PB}{\mathbb{P}\mathbf{B}} %projective Borel

\newcommand{\PBplus}{\mathbb{P}B^{+}} %projective version of B_n^+, without the index n

\newcommand{\phee}{\varphi} %phi
\newcommand{\veps}{\varepsilon} %epsilon
\newcommand{\barra}[1]{\overline{#1}}

\newcommand{\gera}[1]{\langle {#1} \rangle}
\newcommand{\set}[1]{\{ #1 \}}
\newcommand{\mb}{\mathbf}
\newcommand{\mc}{\mathcal}
\newcommand{\mf}{\mathfrak}

\newcommand{\into}{\hookrightarrow}
\newcommand{\onto}{\twoheadrightarrow}
\newcommand{\nsgp}{\trianglelefteq} %normal subgroup
\newcommand{\Ri}{R_{\infty}} % R_infty
\definecolor{amethyst}{rgb}{0.6, 0.4, 0.8}

\DeclareMathOperator{\Aut}{Aut}

\DeclareMathOperator{\rk}{rk}

\DeclareMathOperator{\carac}{char}
\title[Twisted conjugacy in soluble arithmetic groups]{Twisted conjugacy in soluble arithmetic groups}
\author{Paula Macedo Lins de Araujo and Yuri Santos Rego}
\address{University of Lincoln, \newline 
Charlotte Scott Centre for Algebra, \newline 
Isaac Newton Building, Brayford Pool, \newline 
LN6 7TS Lincoln, United Kingdom}
\email{pmacedolinsdearaujo@lincoln.ac.uk}
\address{Otto-von-Guericke-Universit\"at Magdeburg, \newline 
Fakult\"at f\"ur Mathematik -- Institut f\"ur Algebra und Geometrie, \newline 
Postfach 4120, 39016 Magdeburg, Deutschland}
\email{ysantosrego@lincoln.ac.uk}
\curraddr{Charlotte Scott Research Centre for Algebra, University of Lincoln, Isaac Newton Building, Brayford Pool, LN6 7TS Lincoln, United Kingdom}

\begin{document}
\begin{abstract}
Reidemeister numbers of group automorphisms encode the number of twisted conjugacy classes of groups and might yield information about self-maps of spaces related to the given objects. Here we address a question posed by Gon\c{c}alves and Wong in the mid 2000s: we construct an infinite series of compact connected solvmanifolds (that are \textit{not} \textit{nil}manifolds) of strictly increasing dimensions and all of whose self-homotopy equivalences have vanishing Nielsen number. To this end, we establish a sufficient condition for a prominent (infinite) family of soluble linear groups to have the so-called property~$R_\infty$. In particular, we generalize or complement earlier results due to Dekimpe, Gon\c{c}alves, Kochloukova, Nasybullov, Taback, Tertooy, Van den Bussche, and Wong, showing that many soluble $S$-arithmetic groups have~$R_\infty$ and suggesting a conjecture in this direction.  
\end{abstract}
\maketitle \vspace{-1.0cm}
\thispagestyle{empty}

\section{Introduction}
Since at least the 1930s, twisted conjugation has had its presence in Lie theory~\cite{SpringerSurvey} and in parallel in fixed point theory after Lefschetz, Reidemeister, and Nielsen~\cite{BobBrownBook,GeogheganFix,JiangPrimer}. 
Let $G$ be a group with a given automorphism $\phee\in\Aut(G)$. The orbits of the action $G \times G \to G$, $(h,g) \mapsto h g \phee(h)^{-1}$ are called $\phee$\emph{-(twisted) conjugacy classes} or \emph{Reidemeister classes} of $\phee$.
The total number of $\phee$-conjugacy classes is the \emph{Reidemeister number} of $\phee$, denoted by $R(\phee)$. 
In the terminology introduced by Taback and Wong~\cite{TabackWong0}, the group $G$ has \emph{property}~$\Ri$ if $R(\phee) = \infty$ for all $\phee \in \Aut(G)$.

After influential works across different areas (e.g., \cite{Zappa0, SteinbergEndo, FelshtynHill, LevittLustig}), many authors set out to establish which groups have property~$\Ri$. Whether $G$ has~$\Ri$ or admits a specific $\phee\in\Aut(G)$ with $R(\phee)<\infty$ may shed some light on the structure of $G$ or on properties of spaces related to it. Let us mention some recent examples. If $G$ has~$\Ri$ and a group $\Gamma$ algebraically fibers with kernel $G$, then there is no compact model for the classifying space $\underline{\underline{B}}\Gamma$ of $\Gamma$ with respect to the family of virtually cyclic subgroups \cite{TimmXiaolei0}. In another direction, if $G$ is polycylic and admits $\phee \in \Aut(G)$ with $R(\phee^n) < \infty$ for all $n \in \Z_{\geq 1}$, then $G$ must be virtually nilpotent \cite{JonasStronglySI}. Yet in a different vein, if there is a nonzero global fixed point for the canonical action of $\Aut(G)$ on the first cohomology $H^1(G,\R)$, then $G$ has~$\Ri$ \cite{DesiDaciberg,Brasuca0}.  

Turning to the topological origins of twisted conjugacy, we recall two central problems in fixed point theory. Firstly, one is interested in computing --- or at least estimating --- the Nielsen number $N(f)$ of a continuous self-map $f : M \to M$ of a (compact, connected) space $M$. The Nielsen number $N(f)$ is a homotopy invariant that gives the number of essential fixed point classes of $f$, though the challenge is that $N(f)$ can be notoriously difficult to determine. Secondly, one is interested in constructing spaces whose self-maps have prescribed numbers of fixed points. Now, in case $G \cong \pi_1(M)$, the fundamental group of a (compact, connected) manifold $M$, the Reidemeister numbers of automorphisms of $G$ are related to fixed-point properties of continuous self-maps of $M$; we refer the reader to the collection~\cite{HandbookFPT} for a broad overview on fixed point theory. 

In the above setting, Gon\c{c}alves and Wong gave in the influential paper~\cite{DacibergWongCrelle} the first examples of nilmanifolds (in every dimension $n \geq 4$) all of whose self-homeomorphisms have Nielsen number zero (and thus are homotopic to fixed-point-free maps). Given their methods and known results from the literature, they ponder what can be said in the general case of solvmanifolds \cite[p.~24]{DacibergWongCrelle}. Since then, few sporadic examples of solvmanifolds (that are not nilmanifolds) whose self-maps have vanishing Nielsen number were constructed, and only in dimensions up to $4$ \cite{DacibergWongCounterexample, FelshtynLeeInfraSolv, JoLeeInfraSolv, KarelSamIrisSolv, KarelIrisInfrasolv}. In the present work, we solve the problem of Gon\c{c}alves--Wong to some extent. More precisely, we construct solvmanifolds of arbitrarily high dimensions --- that are not nilmanifolds --- and all of whose self-homotopy equivalences have vanishing Nielsen number. 

\begin{thm} \label{thm:geometricapplication}
There exists a family of closed connected solvmanifolds 
\[ M_1 \subset M_2 \subset M_3 \subset \cdots\] 
of dimensions $\dim(M_n) = n^2+2n$ and with the following properties.
\begin{enumerate}
    \item The Nielsen number of any self-homotopy equivalence of $M_n$ is zero. (Hence any such map on $M_n$ is homotopic to a fixed-point-free map.) In case $n \geq 2$ and $h : M_n \to M_n$ is a homeomorphism, then $h$ is isotopic to a fixed-point-free map. 
    \item The fundamental group $\pi_1(M_n)$ is of exponential growth and of derived length $\lceil \log_2(n+1) \rceil+1$. (In particular, none of the manifolds $M_n$ are nilmanifolds, and $\pi_1(M_n)$ is not metabelian if $n \geq 2$.)
\end{enumerate}
\end{thm}

The manifolds from \cref{thm:geometricapplication} are quite explicit (under a suitable interpretation). The first space $M_1$ is illustrated in Figure~\ref{fig:oneandonly}. 

\newsavebox{\NiceTorus}
\savebox{\NiceTorus}{
\begin{tikzpicture}%[rotate=180]
%Torus
\filldraw[fill=gray!50] (0,0) ellipse (1.6 and .9);
%Hole
\begin{scope}[scale=.8]
\filldraw[rounded corners=24pt,fill=white] (-.925,0)--(0,-.5)--(.925,0) (-.9,0)--(0,.6)--(.9,0);
\end{scope}
%Cut 1
\draw[densely dashed] (0,-.9) arc (270:90:.2 and .365);
\draw (0,-.9) arc (-90:90:.2 and .365);
%Cut 2
%\draw (0,.9) arc (90:270:.2 and .348);
%\draw[densely dashed] (0,.9) arc (90:-90:.2 and .348);
\end{tikzpicture}
 }

\newsavebox{\EsseUm}
\savebox{\EsseUm}{
\begin{tikzpicture}
\node (1) at (0,0) {$\mathbb{S}^1$};
\draw [very thick, black] (0,0) circle (1.0);
\end{tikzpicture}
}

\begin{figure}[htb!] \label{fig:oneandonly}
\begin{tikzpicture}[every node/.append style={transform shape},scale=1]
\node (T1) {\usebox{\NiceTorus}};
\node (M) [right=of T1] {$M_1$};
\node (C) [right=of M] {\usebox{\EsseUm}};
\draw[arrows = {Hooks[right]->}] (T1) to (M);
\draw[arrows = {->>}] (M) to (C);
 \end{tikzpicture}
  \caption{$M_1$ as a fiber bundle over the circle: $M_1$ is the $3$-manifold given by the mapping torus of the self-homeomorphism of $(\R/\Z)^2$ induced by the Anosov map $\left(\begin{smallmatrix} 0 & 1 \\ 1 & 1 \end{smallmatrix} \right)$.}
 \end{figure}

The main new ingredient in the proof of \cref{thm:geometricapplication} is our main technical result, \cref{thm:Aff} below, which gives a contribution to the theory of twisted conjugacy of linear groups. To contextualize, a problem of particular interest has been that of determining which finitely generated linear groups have property~$\Ri$, especially $S$-arithmetic groups. Among such groups whose twisted conjugacy classes were studied, those that are full-sized (i.e., contain nonabelian free subgroups) turn out to have property~$\Ri$; see, for example, \cite{LevittLustig, FelshtynHyperbolic, MubeenaSankaranLattices, MubeenaSankaranQI, Timur0, MitraSankaranGL2, MitraSankaran}. Conjecturally, any finitely generated full-sized linear group should have~$\Ri$ \cite[Conjecture~R]{FelshtynTroitskyAspects}. 

In contrast, there are multiple examples of (virtually) soluble $S$-arithmetic groups that do, and many that do not, have property~$\Ri$; see \cite{DacibergWongCounterexample, DacibergWongWreath, TabackWong0, DacibergWongCrelle, FelshtynDacibergMetabelian, KarelPenni, KarelKaiserSam, KarelDacibergNil, KarelSamIrisSolv, TimurUniTri}. A particular issue in this setting is that there are few recognizable patterns regarding `what causes' such groups to exhibit~$\Ri$. 

Here we provide a tool to check for property~$\Ri$ in a wide class of soluble matrix groups, showcasing how the base ring of coefficients plays a role in our case. These groups arise from (variations of) the $\Z$-subschemes $\mbB_n \leq \GL_n$ of (upper) triangular matrices and $\Aff \leq \GL_2$ of one-dimensional affine transformations. Given a (commutative, unital) ring $R$, we work concretely with the groups of $R$-points 
\[ 
\mbB_n(R) = \left( \begin{smallmatrix} * & * & \cdots & * \\  & * & & \\ & & \ddots & \vdots \\ & & & * \\  \end{smallmatrix} \right) \leq \GL_n(R) 
\quad \text{ and } \quad 
\Aff(R) = \begin{pmatrix} * & * \\ 0 & 1 \end{pmatrix} \leq \GL_2(R),
\]
and also their variants $B_n^+(R)$ and $\Aff^+(R)$ without torsion on the diagonal and their projective quotients $\PB_n(R)$ and $\PBplus_n(R)$. The reader is referred to \cref{sec:osgrupos} for more on $\mbB_n(R)$, $\Aff(R)$, and their variants $\PB_n(R)$, $B_n^+(R)$, $\Aff^+(R)$, $\PBplus_n(R)$. We prove the following.

\begin{thm} \label{thm:Aff}
Let $R$ be an integral domain and let $A \in \{\Aff,\Aff^+\}$. Every $\psi \in \Aut(A(R))$ induces an automorphism $\psi_{\mathrm{tf}}$ on the torsion-free part of the diagonal subgroup $\left(  \begin{smallmatrix} * & 0 \\ 0 & 1 \end{smallmatrix} \right) \cong A(R) / \left(  \begin{smallmatrix} 1 & * \\ 0 & 1 \end{smallmatrix} \right)$ such that $R(\psi) \geq R(\psi_{\mathrm{tf}})$. Moreover, if $R^\times$ is finitely generated and every such $\psi_{\mathrm{tf}}$ has infinitely many fixed points, then $G(R)$ has property~$\Ri$ for \emph{all} $n \geq 2$ and \emph{all} $G \in \{ \mbB_n, \PB_n, \Aff \}$ (in case $A = \Aff$), respectively \emph{all} $G \in \{ B_n^+, \PBplus_n, \Aff^+ \}$ (in case $A = \Aff^+$).
\end{thm}

While the groups to which \cref{thm:Aff} applies might seem a bit restrictive at first glance, we recall that they are related to, generalize, or provide analoga of important families of $S$-arithmetic metabelian groups appearing in the literature. Let us illustrate this by providing some examples.

We start with the work of Taback--Wong~\cite{TabackWong0} which strongly influenced ours. They consider `generalized soluble Baumslag--Solitar' groups $G_n$ as follows. Let $p_1, \ldots, p_k$ be the the prime factors of~$n \in \N$. Then $G_n$ is defined by the presentation 
 \[G_n \cong \langle a,t_1, \dots, t_k \mid t_it_j=t_jt_i, \, t_{i}at_{i}^{-1}=a^{p_i} \rangle.\]
Their groups $G_n$ coincide with $\Aff^+(\Z[1/n])$. 
For a polycyclic example, some of the groups $\Z^2 \rtimes_\theta \Z$ investigated by Gon\c{c}alves--Wong in~\cite{DacibergWongCounterexample} are of the form $\Aff^+(\ri)$ for appropriate rings of integers $\ri$ of quadratic number fields. And, for an example in positive characteristic, the groups $\Aff^+(\F_p[t,t^{-1},(t-1)^{-1}])$ are quasi-isometric to certain Diestel--Leader graphs~\cite[Section~3]{BartholdiNeuhauserWoessDL}; see also \cite{DesiSaidSSFPn}. In particular, applying \cref{thm:Aff} to such groups with $p > 2$, we obtain a new series of finitely presented finite-state self-similar groups with property~$\Ri$ but not of homological type $\mathtt{FP}_\infty$; see \cref{pps:zft}.

Although we are mostly interested in $S$-arithmetic groups, no familiarity with these objects is assumed. In a companion paper~\cite{Bn2}, which was part of a previous version of the present work, we further discuss the problem of classifying which soluble $S$-arithmetic groups have~$\Ri$. There, we give an alternative criterion to check for~$\Ri$ in $\mbB_n(R)$ and $\PB_n(R)$, also providing explicit examples of non-$S$-arithmetic groups to which that tool applies. In \cref{thm:Aff} here, the problem of checking for~$\Ri$ for $\mbB_n(R)$ and its variants is reduced to looking at automorphisms of $\Aff(R) \cong R\rtimes R^\times$, while the criterion from~\cite{Bn2} reduces the problem to looking at ring automorphisms of~$R$. This is a further complexity reduction since $\anel{\Aut}(R)$ embeds naturally into $\Aut(G(R))$ for $G \in \{\mbB_n, \PB_n\}$. The drawback is that the tool from~\cite{Bn2} works for solubility class $n \geq 4$ and in general does not apply to the variants without torsion on the diagonal.

The present article is divided into two parts, mostly independent. The first part starts in \cref{sec:geometricIntro} with some more background and related work from fixed point theory. Assuming our technical \cref{thm:Aff} to hold and postponing its proof, we illustrate in \cref{illustrating} how to apply it concretely over a certain ring of integers. Using this result, we give in \cref{ProofApplication} a proof of our main application, \cref{thm:geometricapplication}. \cref{sec:geometricIntro} is written as self-contained as possible, so that the reader interested in the topological aspects of the paper can skip directly to it. However, readers interested in the algebraic aspects of this work might benefit from reading Sections~\cref{illustrating} and~\ref{ProofApplication} since they serve as a warm-up illustrating many of the objects and arguments that appear later throughout the paper. 

The second part comprises of Sections~\ref{sec:LemmataRinfty}, \ref{sec:provadothmA} and~\ref{Applications}, with many results generalizing what was seen in \cref{sec:geometricIntro}. In \cref{sec:LemmataRinfty} we collect known facts about Reidemeister numbers needed for our work. We fix notation and recall standard properties of our groups of interest in \cref{sec:osgrupos}. In \cref{sec:IsUncharacteristic}, we address the question of whether the group $\mbU_n(R)$ of unitriangular matrices is characteristic in the groups we consider; cf. Propositions~\ref{pps:UnNOTcharacteristic} and~\ref{pps:char}. \cref{thm:Aff} is more precisely restated and proved in \cref{sec:provadothmA}.  Following that, we provide more applications of \cref{thm:Aff} to families of $S$-arithmetic groups in \cref{Applications}, widely generalizing the arithmetic example used to construct our solvmanifolds; cf. \cref{pps:zft}. In that section we also discuss the current state of knowledge and pose a conjecture about~$\Ri$ for soluble $S$-arithmetic groups; see \cref{theconjecture}.

\section{Solvmanifolds and fixed point theory} \label{sec:geometricIntro}

 The Nielsen number $N(f)$ of a self-map $f : M \to M$ of a (compact, connected) manifold $M$ is a powerful homotopy invariant that provides the minimal number of fixed points for maps homotopic to $f$; see~\cite{JiangPrimer,Wecken}. As is known~\cite{HandbookFPT}, computing Nielsen numbers can be very difficult. In the 1970s, Brooks, Brown, Pak and Taylor proved that $n$-dimensional tori are the only (compact, connected) Lie groups whose self-maps have Nielsen number equal to the absolute value of their (easily computable) Lefschetz number~\cite{BBPT}. 

Similar relations using Lefschetz or Reidemeister numbers were later shown to hold for a few other types of manifolds~\cite{HandbookFPT}. Most prominently, a theorem of Anosov~\cite{AnosovNil} proved in the 1980s implies an alternative of the Brooks--Brown--Pak--Taylor result using Reidemeister numbers. That is, one has the following `Anosov-type' theorem: if $M$ is a compact connected \emph{nil}manifold, $f : M \to M$ a self-homotopy equivalence and $\phee \in \Aut(\pi_1(M))$ the automorphism induced by $f$ on the fundamental group $\pi_1(M)$, then $N(f) = 0$ in case $R(\phee) = \infty$, otherwise $N(f) = R(\phee) < \infty$; cf.~\cite{WongCrelle,DacibergWongCrelle} and references therein. Around two decades later, Gon\c{c}alves and Wong explored this relationship and property~$\Ri$ and managed to construct, in each dimension $n \geq 4$, examples of $n$-dimensional nilmanifolds whose self-homeomorphisms can be homotoped to become free of fixed points~\cite[Theorem~6.1]{DacibergWongCrelle}. 

All of the above spaces belong more generally to the class of \emph{solv}manifolds, i.e., homogeneous spaces of soluble Lie groups modulo closed subgroups. In that fundamental paper~\cite{DacibergWongCrelle}, Gon\c{c}alves and Wong broadly pose the question of what happens with infra-nilmanifolds and with solvmanifolds in general~\cite[p.~24]{DacibergWongCrelle}. Since then, there has been substantial progress towards determining which (virtually) {nilpotent} groups have property~$\Ri$ \cite{Romankov,KarelPenni,KarelDacibergNil,TimurUniTri}, which then led to further constructions of (infra-)nilmanifolds --- again in arbitrarily high dimensions --- with plenty of fixed-point-free self-maps; see, for example, \cite[Corollary~3.6]{KarelDacibergSurface}. 

In contrast, the question of Gon\c{c}alves and Wong remained elusive for general solvmanifolds not covered by the previous cases. The problem posed by them boils down to the following: do there exist solvmanifolds (that are \emph{not} nilmanifolds) in arbitrarily high dimensions all of whose self-homotpy equivalences can be made fixed-point-free? So far, this has been investigated in isolated cases in dimensions up to~$4$; cf.~\cite{DacibergWongCounterexample, FelshtynLeeInfraSolv, JoLeeInfraSolv, KarelSamIrisSolv, KarelIrisInfrasolv}. 

Our \cref{thm:geometricapplication} thus answers the above question in the affirmative. Unsurprisingly, the starting dimension $3$ in \cref{thm:geometricapplication} is optimal. Indeed, this is the lowest dimension for which Wecken's theorem applies~\cite{Wecken}, and moreover the $2$-dimensional case --- i.e., the Klein bottle $K$ --- has to be trivially excluded because $K$ {does} have self-homotopy equivalences with positive Nielsen number (see \cite[Example~8.8]{KarelIrisInfrasolv}), even though $\pi_1(K)$ has property~$\Ri$ \cite[Theorem~2.2]{DacibergWongCrelle}. 

The proof of \cref{thm:geometricapplication} has two main ingredients. The first is a version of the `Anosov-type' theorem --- originally due to Keppelmann and McCord~\cite{KeppelmannMcCordSolv} and improved by Dekimpe and Van den Bussche~\cite{KarelIrisInfrasolv}; cf. \cref{thm:KMcDVdB} below. The second ingredient is the construction of (torsion-free, non-virtually nilpotent) polycyclic groups of large solubility class (and increasing Hirsch lengths) and with property~$\Ri$, which is a consequence of our technical \cref{thm:Aff} applied to a specific arithmetic group. The construction of the desired manifolds is done via (generalizations of) the classical Mal'cev completion; see~\cite{MostowSolv,AuslanderSolvI,Baues,JonasNIL}. We remark that Gon\c{c}alves and Kochloukova established a cohomological criterion to check for property~$\Ri$ in a wide class of nilpotent-by-abelian groups; cf. \cite[Theorem~4.3]{DesiDaciberg}. However, their results do not apply to polycyclic groups, hence cannot imply the existence of solvmanifolds as those of our \cref{thm:geometricapplication}. 

Despite our contribution towards the question of Gon\c{c}alves--Wong, we do not know whether one can construct {in every dimension} $n \geq 3$ a solvmanifold with the properties prescribed in \cref{thm:geometricapplication}. Another potential source of examples of higher dimensional solvmanifolds with vanishing Nielsen numbers was indicated by Gon\c{c}alves and Wong themselves in an earlier paper~\cite{DacibergWongCounterexample}. They gave conditions for extensions of the form $\Z^2 \rtimes_A \Z$ to have~$\Ri$ depending on the matrix $A \in \SL_2(\Z)$ --- the natural next step would be to generalize their methods to $\Z^n \rtimes_{A_n} \Z$ with $A_n \in \GL_n(\Z)$. However, the polyclic group $\Z^n \rtimes_{A_n} \Z$ might well be nilpotent (so that the corresponding solvmanifold is in fact a nilmanifold), and $\Z^n \rtimes_{A_n} \Z$ is \emph{metabelian}, i.e., soluble of class $2$ --- thus groups in this family do not have increasing solubility class. As exemplified in work of Dekimpe--Tertooy--Van den Bussche \cite{KarelSamIrisSolv}, moving from $\Z^2 \rtimes_{A_2} \Z$ up to the four-dimensional case $\Z^3 \rtimes_{A_3} \Z$ is a nontrivial matter, so that heuristically it might be difficult to classify which extensions $\Z^n \rtimes_{A_n} \Z$ have property~$\Ri$.

\subsection{A first step towards Theorem~\ref{thm:geometricapplication}: illustrating how Theorem~\ref{thm:Aff} works}\label{illustrating}

One necessary step for the proof of \cref{thm:geometricapplication} is showing that a certain infinite family of matrix groups have~$\Ri$. We therefore take the opportunity to illustrate how \cref{thm:Aff} works in practice. We keep this section as self-contained and independent from the rest of this work as possible, simply assuming that \cref{thm:Aff} holds and using basic facts from algebraic number theory~\cite{Neukirch}. 

Let $\phi$ denote the (large) golden ratio, that is, $\phi = \frac{1+\sqrt{5}}{2}$ is the positive root of the polynomial $x^2-x-1 \in \Z[x]$. The ring of integers $\Z[\phi] \subset \Q(\phi) = \Q(\sqrt{5})$ admits $\{1,\phi\}$ as an integral basis~\cite[p.~16]{Neukirch}, i.e., 
\[\Z[\phi] = \Z \oplus \Z\phi = \{ a + b\phi \mid a,b \in \Z \} \cong \Z^2.\] 
The (small, or inverse) golden ratio $\tau = \frac{-1+\sqrt{5}}{2} = \phi - 1$ is the multiplicative inverse of $\phi$ in $\Z[\phi]$. Note that $\phi$ is a torsion-free unit and, by Dirichlet's unit theorem~\cite[{\textsection}~7]{Neukirch}, the group of units of $\Z[\phi]$ is given by 
\[\Z[\phi]^\times = \{ \pm \phi^n \mid n \in \Z \} = \gera{-1,\phi} \cong C_2 \times \Z.\]
In particular, $\{\phi^n \mid n \in \Z\} = \gera{\phi} \cong \Z$ a set of torsion-free units of $\Z[\phi]$.

The corresponding group of `positive' affine transformations $\Aff^+(\Z[\phi])$ is given as
\[
\Aff^+(\Z[\phi]) = \left\{\begin{pmatrix} \diamond & * \\ 0 & 1 \end{pmatrix} \, \middle| \, \diamond \in \gera{\phi}, \ast \in \Z[\phi] \right\} \leq \GL_2(\Z[\phi]),
\]
that is, $\Aff^+(\Z[\phi])$ is the subgroup of $\Aff(\Z[\phi])$ of elements whose diagonal entries belong the above mentioned set of torsion-free units.

Let $B_n^{+}(\Z[\phi]) \leq \GL_n(\Z[\phi])$ denote the subgroup of upper triangular matrices with only positive entries on the main diagonal. (Thus $a_{i,i} = \phi^{m_i} \in \gera{\phi} \subset \Z[\phi]^\times$ for any matrix $A = (a_{i,j}) \in B_n^+(\Z[\phi])$.) Consider the group $\PBplus_n(\Z[\phi])$, which is the quotient of $B_n^{+}(\Z[\phi])$ by its center 
\[Z_n(\gera{\phi}) := \{ \mathrm{Diag}(\phi^{\veps}, \ldots, \phi^{\veps}) \mid \veps \in \Z \},\] 
which is composed of all scalar matrices with entries in $\gera{\phi} \subset \Z[\phi]^\times$. 

\begin{pps}\label{pps:goldenRi} 
The groups $\PBplus_n(\Z[\phi])$ and $B_n^+(\Z[\phi])$ have property~$\Ri$ for all $n\geq 2$.
\end{pps}

\begin{proof} 
First, a remark: direct matrix computations show that the unitriangular subgroup $\left(\begin{smallmatrix} 1 & * \\ 0 & 1 \end{smallmatrix} \right) \leq \Aff^+(\Z[\phi])$ is normal and in fact characteristic since it is the maximal abelian normal subgroup of $\Aff^+(\Z[\phi])$. The quotient $\Aff^+(\Z[\phi]) / \left(\begin{smallmatrix} 1 & * \\ 0 & 1 \end{smallmatrix} \right)$ is easily seen to be (isomorphic to) the diagonal subgroup $\left\{ \left(\begin{smallmatrix} \phi^\ell & 0 \\ 0 & 1 \end{smallmatrix} \right) \mid \ell \in \Z\right\}$ of $\Aff^+(\Z[\phi])$. We write $\left(\begin{smallmatrix} a & b \\ 0 & 1 \end{smallmatrix} \right) \cdot \left(\begin{smallmatrix} 1 & * \\ 0 & 1 \end{smallmatrix} \right)$ for (left) cosets of $\Aff^+(\Z[\phi])$ modulo the unitriangular subgroup $\left(\begin{smallmatrix} 1 & * \\ 0 & 1 \end{smallmatrix} \right)$. The proof now will be a straightforward application of \cref{thm:Aff}, which works as follows:  
\begin{enumerate}
\item Consider an arbitrary automorphism $\psi \in \Aut(\Aff^{+}(\Z[\phi]))$. This map induces an automorphism $\psi_{\mathrm{tf}}$ on the quotient $\Aff^+(\Z[\phi]) / \left(\begin{smallmatrix} 1 & * \\ 0 & 1 \end{smallmatrix} \right)$. In symbols, $\psi_{\mathrm{tf}}$ is the automorphism defined by 
\[\psi_{\mathrm{tf}}\left( \left(\begin{smallmatrix} \phi^\ell & 0 \\ 0 & 1 \end{smallmatrix} \right) \cdot \left(\begin{smallmatrix} 1 & * \\ 0 & 1 \end{smallmatrix} \right) \right) = \psi\left( \left(\begin{smallmatrix} \phi^\ell & 0 \\ 0 & 1 \end{smallmatrix} \right) \right) \cdot \left(\begin{smallmatrix} 1 & * \\ 0 & 1 \end{smallmatrix} \right).  \]
	\item If $\psi_{\mathrm{tf}}$ fixes infinitely many points, then \cref{thm:Aff} says that all groups $\PBplus_n(\Z[\phi])$ and $B_n^+(\Z[\phi])$ with $n\geq 2$ have~$\Ri$. 
\end{enumerate} 
By definition, there exist $m \in \Z \setminus \{0 \}$ and $x_{\phi} \in \Z[\phi]$ such that 
\[\psi \left(\left( \begin{smallmatrix} \phi & 0 \\  0 & 1 \\  \end{smallmatrix} \right)\right) = \left( \begin{smallmatrix} \phi^m & x_\phi \\  0& 1 \\  \end{smallmatrix} \right) = \left( \begin{smallmatrix} \phi^m & 0 \\  0& 1 \\  \end{smallmatrix} \right) \left( \begin{smallmatrix} 1 \phantom{||} & \phi^{-m} x_\phi \\  0 \phantom{||} & 1 \\  \end{smallmatrix} \right).\] 
In particular, 
\[\psi_{\mathrm{tf}} \left(\left( \begin{smallmatrix} \phi & 0 \\  0 & 1 \\  \end{smallmatrix} \right)\cdot \left(\begin{smallmatrix} 1 & * \\ 0 & 1 \end{smallmatrix} \right) \right)= \left( \begin{smallmatrix} \phi^m & 0 \\  0& 1 \\  \end{smallmatrix} \right) \cdot \left(\begin{smallmatrix} 1 & * \\ 0 & 1 \end{smallmatrix} \right).\] 
Thus, if we can show that $m = 1$, it will follow that the element $\left( \begin{smallmatrix} \phi & 0 \\  0 & 1 \\  \end{smallmatrix} \right)\cdot \left(\begin{smallmatrix} 1 & * \\ 0 & 1 \end{smallmatrix} \right)$ is fixed by $\psi_{\mathrm{tf}}$, hence also all of its powers. As $\phi$ has infinite order, this gives infinitely many points fixed by $\psi_{\mathrm{tf}}$, which will finish off the proof.

We argue using straightforward matrix computations and the algebraic relation $\phi^2 = \phi + 1$ that $m = 1$. Since $\left(\begin{smallmatrix} 1 & * \\ 0 & 1 \end{smallmatrix} \right)$ is a characteristic subgroup of $\Aff^+(\Z[\phi])$, there exists an $r \in \Z[\phi]$ such that 
\[\psi \left(\left( \begin{smallmatrix} 1 & 1 \\  0 & 1 \\  \end{smallmatrix} \right)\right)=\left( \begin{smallmatrix} 1 & r \\  0& 1 \\  \end{smallmatrix} \right).\] 
On the one hand, we have 
\begin{align*} 
\psi\left(\left( \begin{smallmatrix} 1 & \phi^2 \\  0& 1 \\  \end{smallmatrix} \right)\right) 
& = \psi\left(\left( \begin{smallmatrix} \phi & 0 \\  0& 1 \\  \end{smallmatrix} \right)^2 \, \left( \begin{smallmatrix} 1 & 1 \\  0& 1 \\  \end{smallmatrix} \right) \, \left( \begin{smallmatrix} \phi & 0 \\  0& 1 \\  \end{smallmatrix} \right)^{-2}\right) \\	
& = \left( \begin{smallmatrix} \phi^m & x_\phi \\  0& 1 \\  \end{smallmatrix} \right)^2 \, 
\left( \begin{smallmatrix} 1 & r \\  0& 1 \\  \end{smallmatrix} \right) \, \left( \begin{smallmatrix} \phi^m & x_\phi \\  0& 1 \\  \end{smallmatrix} \right)^{-2} = \left(\begin{smallmatrix} 1 & \phi^{2m}r \\  0& 1 \\  \end{smallmatrix} \right).
\end{align*}
One the other hand, as $\phi^2=\phi+1$, we have 
\begin{align*} 
\psi\left(\left(\begin{smallmatrix} 1 &\phi^{2} \\  0& 1 \\  \end{smallmatrix} \right)\right) 
& = \psi\left(\left(\begin{smallmatrix} 1 & \phi \\  0& 1 \\  \end{smallmatrix} \right) \left(\begin{smallmatrix} 1 & 1 \\  0& 1 \\  \end{smallmatrix} \right)\right) \\	
& = \psi\left(\left(\begin{smallmatrix} \phi & 0\\  0& 1 \\  \end{smallmatrix} \right) \left(\begin{smallmatrix} 1 & 1 \\  0& 1 \\  \end{smallmatrix}\right) \left(\begin{smallmatrix} \phi & 0 \\  0& 1 \\  \end{smallmatrix} \right)^{-1}\right) 
\psi\left(\left(\begin{smallmatrix} 1 & 1 \\  0& 1 \\  \end{smallmatrix} \right)\right)\\
& = \left(\begin{smallmatrix} \phi^m & x_\phi \\  0& 1 \\  \end{smallmatrix} \right) \left(\begin{smallmatrix} 1 & r \\  0& 1 \\  \end{smallmatrix}\right) \left(\begin{smallmatrix} \phi^m & x_\phi \\  0& 1 \\  \end{smallmatrix} \right)^{-1} 
\left(\begin{smallmatrix} 1 & r \\  0& 1 \\  \end{smallmatrix} \right)  = \left(\begin{smallmatrix} 1 & \phi^m r + r \\  0& 1 \\  \end{smallmatrix} \right).
\end{align*}
Comparing both values for $\psi\left(\left(\begin{smallmatrix} 1 &\phi^{2} \\  0& 1 \\  \end{smallmatrix} \right)\right) $, we conclude that $\phi^{2m}r=\phi^{m}r+r$. Since $\psi$ is an automorphism, $r$ must be nonzero. As $\Z[\phi]$ is an integral domain, it follows that $\phi^{2m}=\phi^m+1$. Thus $\phi^m$ is also a positive real root of $x^2-x-1 \in \Z[x]$. Since the two roots of this polynomial are $\phi > 0$ and $\nu = \frac{1-\sqrt{5}}{2} < 0$, the above is only possible if $m = 1$, concluding the proof. 
\end{proof}

\subsection{Proof of Theorem~\ref{thm:geometricapplication}}\label{ProofApplication}
Throughout this section, for $n \geq 1$ let $\Gamma_n$ denote the group $\PBplus_{n+1}(\Z[\phi])$ considered in \cref{illustrating}. We construct the manifolds $M_n$ using the groups $\Gamma_n$ and classic results from differential geometry~\cite{MostowSolv, WangSolv, AuslanderSolvI, Wolf, Baues, JonasNIL}. 

Recall that $B_{n+1}^+(\Z[\phi])$ is the subgroup of upper triangular matrices in $\GL_{n+1}(\Z[\phi])$ whose diagonal entries belong to the set of positive units $\Z\cong \gera{\phi} \subset \Z[\phi]^\times$. The group $\Gamma_{n+1} = \PBplus_{n+1}(\Z[\phi])$ is the quotient of $B_{n+1}^+(\Z[\phi])$ modulo scalar matrices in $B_{n+1}^+(\Z[\phi])$. In particular, the subgroup $\mbU_{n+1}(\Z[\phi]) \leq \GL_{n+1}(\Z[\phi])$ of upper \emph{uni}triangular matrices (i.e., upper triangular matrices with $1$'s on the main diagonal) naturally embeds as a normal subgroup in $\Gamma_n$. Moreover, the corresponding quotient is isomorphic to the diagonal subgroup of $\PBplus_{n+1}(\Z[\phi])$, which is seen to consist of $n+1$ (commuting) copies of $\gera{\phi} \cong \Z$ modulo scalar matrices with entries in $\gera{\phi} \cong \Z$. 

The above considerations imply that $\Gamma_n$ fits into a short exact sequence 
\begin{equation} \label{eq:sesGamman}
1 \into \mbU_{n+1}(\Z[\phi]) \into \Gamma_n \onto \Z^n \onto 1.
\end{equation}
As is known (or can be checked by matrix computations), the unitriangular group $\mbU_{n+1}(\Z[\phi])$ is nilpotent of class $n$. Since the base ring $\Z[\phi]$ is additively isomorphic to $\Z^2$, the normal subgroup $\mbU_{n+1}(\Z[\phi]) \nsgp \Gamma_n$ is a finitely generated torsion-free nilpotent subgroup. In particular, $\Gamma_n$ is a torsion-free polycyclic group. Invoking a generalization of the Mal'cev completion for soluble groups --- e.g., following Baues~\cite{Baues} or Der\'e~\cite{JonasNIL} --- we obtain a (closed, connected) solvmanifold $M_n$ such that $\pi_1(M_n) \cong \Gamma_n$. Since the above Mal'cev-type construction is functorial and each $\Gamma_n$ naturally embeds into $\Gamma_{n+1}$ by mapping $\PBplus_{n+1}(\Z[\phi])$ to the upper-left corner of $\PBplus_{n+2}(\Z[\phi])$, each $M_n$ so obtained embeds into the succeeding manifold $M_{n+1}$. Moreover, the dimension of $M_n$ equals the Hirsch length of its fundamental group $\pi_1(M_n) \cong \Gamma_n$, which in turn can be read off from the exact sequence~\eqref{eq:sesGamman}. Precisely, the (normal, polycyclic) subgroup $\mbU_{n+1}(\Z[\phi])$ has Hirsch length $2\binom{n}{2}$ and the (polycyclic) quotient $\Z^n$ has Hirsch length $n$, whence 
\[\dim(M_n) = 2\cdot\frac{n(n+1)}{2} + n = n^2 + 2n.\] 
The fact that $\pi_1(M_n) \cong \Gamma_n = \PBplus_{n+1}(\Z[\phi])$ is of exponential growth is a consequence of Wolf's theorem~\cite[Theorem~4.3]{Wolf}. 
The class of solubility of $\Gamma_n$ can be deduced using familiar computations. We sketch this below.

\begin{lem} \label{lem:derivedlengthGamman}
The soluble group $\Gamma_n$ has derived subgroup $\Gamma_n' = \mbU_{n+1}(\Z[\phi])$, and the derived length (i.e., solubility class) of $\Gamma_n$ is exactly $\lceil \log_2(n+1) \rceil +1$.
\end{lem}

\begin{proof}
Fix $n\geq 1$. The upper unitriangular subgroup is generated by \emph{elementary matrices} $e_{i,j}(r)$ with $i < j$ --- that is, the matrix $e_{i,j}(r) \in \GL_{n+1}(\Z[\phi])$ has an off-diagonal $(i,j)$-entry equal to $r \in \Z[\phi]$, all its diagonal entries are $1$, and its entries elsewhere are all zero. Looking back at the decomposition in~\eqref{eq:sesGamman}, one readily checks that $[\Gamma_n, \Gamma_n] \subseteq \mbU_{n+1}(\Z[\phi])$. Now, direct matrix computations yield the following identities in $B_{n+1}(\Z[\phi])$. 
\begin{align*} 
\underbrace{i\text{-th entry}} 
\phantom{,.,1) e_{i,j}(\phi r) (\Diag(1,\ldots,1,\phi,1,\ldots,1))^{-1} e_{i,j}(\phi r)^{-1} =} 
& \\
\Diag(1,\ldots,1,{\phi},1,\ldots,1) e_{i,j}(\phi r) (\Diag(1,\ldots,1,\phi,1,\ldots,1))^{-1} e_{i,j}(\phi r)^{-1} = & \\ 
\Diag(1,\ldots,1,{\phi},1,\ldots,1) e_{i,j}(\phi r) \Diag(1,\ldots,1,\phi^{-1},1,\ldots,1) e_{i,j}(-\phi r) = & \\
= e_{i,j}((\phi^2-\phi)r) = e_{i,j}(r). &
\end{align*}
Projecting the above onto $\Gamma_n = \PBplus_{n+1}(\Z[\phi])$ we have just verified that every generator of $\mbU_{n+1}(\Z[\phi]) \nsgp \Gamma_n$ can be written as a commutator. Thus $\Gamma_n' = [\Gamma_n,\Gamma_n]=\mbU_{n+1}(\Z[\phi])$. 

The derived length of $\Gamma_n$ is thus one plus the derived length of $\mbU_{n+1}(\Z[\phi])$. But the latter is well-known. Indeed, first note that, if $R$ is an integral domain, the derived length and the nilpotency class of $\mbU_{n+1}(R)$ are the same as those of $\mbU_{n+1}(K)$, where $K = \mathrm{Frac}(R)$ denotes the fraction field of~$R$. And over a field $K$ of characteristic zero --- which is the case of $\Q(\sqrt{5}) = \mathrm{Frac}(\Z[\phi])$ --- it is known that the derived length of $\mbU_{n+1}(K)$ is $\lceil \log_2(n+1)\rceil$; see, for instance, \cite[p.~16]{Wehrfritz} or \cite{GlasbySbgps}. (Alternatively, one can adapt a classic argument due to P.~Hall \cite[Section~2]{PHallPgps} and direct computations to check that $\mbU_{n+1}(\Z[\phi])^{(k)} = \gamma_{2^k}(\mbU_{n+1}(\Z[\phi]))$, and conclude from this that the derived series terminates after $\log_2(n+1)$ steps since the central series terminates at $\gamma_{n+2}(\mbU_{n+1}(\Z[\phi])) = 1$; cf. \cite{GlasbySbgps}, noting that the computation of generators of $\mbU_n^{(k)}$ and $\gamma_{k}(\mbU_n)$ over finite fields carries over to arbitrary commutative unital rings.)
\end{proof}

We have thus constructed solvmanifolds fulfilling almost all properties listed in the statement of \cref{thm:geometricapplication}. It remains to check its part~(ii) concerning Nielsen numbers of self-maps. To do so we need the following.

\begin{thm}[{Keppelmann--McCord \cite{KeppelmannMcCordSolv}, Dekimpe--Van den Bussche \cite{KarelIrisInfrasolv}}] \label{thm:KMcDVdB}
Suppose $M$ is a connected compact $\mathcal{NR}$-solvmanifold and let $f : M \to M$ be an arbitrary continuous self-map. Denote by $f_\ast$ the endomorphism of $\pi_1(M)$ induced by $f$. Then the `Anosov relation' holds, that is, 
\[N(f) = \begin{cases} R(f_\ast), & \text{ in case } R(f_\ast) < \infty, \\ 0 & \text{ otherwise}. \end{cases}\]
In particular, if $\pi_1(M)$ has property~$\Ri$, then any continuous self-map of $M$ that induces an automorphism on $\pi_1(M)$ has vanishing Nielsen number.
\end{thm}

In the following, we explain the terminology from \cref{thm:KMcDVdB} and show that each $M_n$ is an $\mathcal{NR}$-solvmanifold. Since $\pi(M_n)=\Gamma_{n}$ has property~$\Ri$ by \cref{pps:goldenRi}, $M_n$ being an $\mathcal{NR}$-solvmanifold will imply that the Nielsen number of any self-homotopy equivalence of $M_n$ is zero. In particular, any such self-map is homotopic to a fixed-point-free map by Wecken's theorem~\cite{Wecken}. Moreover, in case $\dim(M_n) \geq 5$ --- equivalently, $n \geq 2$ --- and $h : M \to M$ is a homeomorphism, a result of Kelly~\cite{KellyIsotopy} shows that $h$ is isotopic to a fixed-point-free map. 
\medskip

Given a group $G$ and a subgroup $H \leq G$, its \emph{isolator} is given by 
\[\sqrt[G]{H}=\{g\in G \mid \exists m\in \N \text{ such that }g^m \in H\}.\] 
Now let $G$ be a torsion-free group that fits into a short exact sequence of the form  
\[1 \into \sqrt[G]{[G,G]} \into G \onto \Z^k \onto 1\] 
such that $N := \sqrt[G]{[G,G]}$ is finitely generated torsion-free nilpotent and $k \in \Z_{\geq 0}$. 
Write $c$ for the nilpotency class of $N$, $\gamma_i(N)$ for the $i$-th term of its lower central series, and $N_i:=\sqrt[N]{\gamma_i(N)}$. One readily checks that the $N_i$ are normal subgroups of $G$ and form a central series of $N$ 
\[1 \nsgp N_c \nsgp \ldots \nsgp N_1=N\] 
whose factors are (finitely generated) free abelian. One can thus define actions of $G/N$ on each factor $N_i/N_{i+1}$, as follows: 
\[\rho_i: G/N \to \Aut\left(N_i/N_{i+1}\right) \quad \text{ is given by } \quad gN \mapsto (xN_{i+1} \mapsto gxg^{-1}N_{i+1}).\]

\begin{dfn} Suppose a group $G$ fits into the set-up above. We say that $G$ has the $\mathcal{NR}$\emph{-property} if, for all $i \in \{1, \ldots, c\}$ and all $gN \in G/N$, the above automorphism $\rho_i(gN): N_i/N_{i+1} \to N_i/N_{i+1}$ has no nontrivial roots of unity as eigenvalues (when $\rho_i(gN)$ is regarded as an element of $\GL_{r_i}(\Z)$, where $r_i = \rk(N_i/N_{i+1})$). 
	
A compact connected solvmanifold $M$ is said to be an $\mathcal{NR}$\emph{-solvmanifold} if its fundamental group $\pi_1(M)$ satisfies the $\mathcal{NR}$-property.
\end{dfn}

Recall that $\Gamma_n = \PBplus_{n+1}(\Z[\phi])$ fits into the short exact sequence  
\begin{equation}\tag{\ref{eq:sesGamman}} 1 \into \mbU_{n+1}(\Z[\phi]) \into \Gamma_n \onto \Z^n \onto 1,\end{equation}
where $\mbU_{n+1}(\Z[\phi])$ is the group of upper unitriangular matrices over $\Z[\phi]$ and $\Z^n \cong \gera{\phi}^n$ is the diagonal part of $\PBplus_{n+1}(\Z[\phi])$, and that $\mbU_{n+1}(\Z[\phi])$ is (finitely generated torsion-free) nilpotent of class $c = n$. In the next lemmata, we compute $N:=\sqrt[\Gamma_n]{{[\Gamma_n,\Gamma_n]}}$ and $N_i=\sqrt[N]{\gamma_i(N)}$. 

\begin{lem}\label{lem:gammazero} With the notation above, $N=\sqrt[\Gamma_n]{[\Gamma_n,\Gamma_n]}=\mbU_{n+1}(\Z[\phi])$. 
\end{lem}
\begin{proof} 
In \cref{lem:derivedlengthGamman} we have already verified that $[\Gamma_n, \Gamma_n] = \mbU_{n+1}(\Z[\phi])$. 
The fact that $\sqrt[\Gamma_n]{\mbU_{n+1}(\Z[\phi])}=\mbU_{n+1}(\Z[\phi])$ follows from the exact sequence~\eqref{eq:sesGamman}.
\end{proof}

\begin{lem}\label{lem:gammas} For each $k \in \{1, \ldots, n\}$, we have $N_k=\gamma_k(\mbU_{n+1}(\Z[\phi]))$. 
\end{lem}
\begin{proof} The case $k=1$ follows from \cref{lem:gammazero}. Fix $k \in \{2, \dots, n\}$. Let us show that, if $\mathbf{y} \notin \gamma_k(\mbU_{n+1}(\Z[\phi]))$, then no power of $\mathbf{y}$ belongs to $\gamma_k(\mbU_{n+1}(\Z[\phi]))$.

One promptly checks that 
an element $\mathbf{x}=(x_{ij})$ of $\gamma_k(\mbU_{n+1}(\Z[\phi]))$ is an $(n+1)\times(n+1)$ matrix whose entries satisfy 
\[ x_{ij} = \begin{cases} 0,    & \text{ if either } i>j \text{ or } i < j < i+k,\\
                           1,     & \text{ if } i=j, 
\end{cases}\]
with no restrictions on the other entries besides belonging to $\Z[\phi]$.

Let us establish an order on the set $\{(i,j) \mid 1 \leq i \leq j \leq n+1\}$. We say that $(i,j) < (m, \ell)$ if either $|j-i| < |\ell-m|$ or if $|j-i| = |\ell-m|$ and $i < m$. 

A matrix $\mathbf{y}=(y_{ij}) \in \mbU_{n+1}(\Z[\phi]) \setminus \gamma_k(\mbU_{n+1}(\Z[\phi]))$ is an upper unitriangular matrix such that $y_{pq} \neq 0$ for some indices $p$ and $q$ satisfying $p < q < p+k$. 
Assume $(p,q)$ is the smallest such pair with respect to the above ordering. We now show that all powers of $\mathbf{y}$ also have nonzero $(p,q)$-entry. 
For each $m \in \N$, denote the entries of $\mathbf{y}^m$ by $y_{ij}^{(m)}$. Then 
\[y_{pq}^{(2)}= \sum_{i=1}^n y_{pi}y_{iq} = \sum_{i=p}^q y_{pi}y_{iq} = 2y_{pq}+\sum_{i=p+1}^{q-1} y_{pi}y_{iq}.\] 
Now, for $p < i <q$, we see that $(p,i) < (p,q)$, hence $y_{pi}=0$. We then conclude $y_{pq}^{(2)} = 2y_{pq} \neq 0$. Moreover, notice that $y_{pj}^{(2)}=0$ for all $j <q$. Using induction, we now show that $y_{pq}^{m}=my_{pq}$ and $y_{pj}^{(m)}=0$ for all $m \in \N$ and $j <q$. In fact, assuming this is true in case $m-1$, we immediately get $y_{pj}^{(m)}=0$ and, moreover, 
\[y_{pq}^{(m)}= \sum_{i=p}^q y_{pi}y_{iq} = y_{pp}^{(m-1)}y_{pq}+y_{pq}^{(m-1)}y_{qq}+\sum_{i=p+1}^{q-1} y_{pi}y_{iq} = my_{pq}.\] 
Since the additive group $\Z[\phi]$ is torsion-free, we get $y_{pq}^{(m)}\neq 0$. For negative powers note that, if $\mathbf{y}^{-1}=(w_{ij})$, then $w_{pq}=-y_{pq} \neq 0$. Thus we can repeat the argument above for $\mathbf{y}^{-1}$.
\end{proof}

We are now ready to show that $\Gamma_n$ has the $\mathcal{NR}$-property. 
\begin{pps} \label{pps:GammasNR}
For each $n \geq 1$, the group $\Gamma_n$ has the $\mathcal{NR}$-property. In particular, the manifolds $M_n$ ($n \geq 1$) are $\mathcal{NR}$-solvmanifolds. 
\end{pps} 
\begin{proof} We need to show that the automorphisms $\rho_i: \Gamma_n/\mbU_{n+1}(\Z[\phi]) \to \Aut(N_i/N_{i+1})$ do not have nontrivial roots of unity as eigenvalues. 
Recall that $\Gamma_n/\mbU_{n+1}(\Z[\phi])$ corresponds to the diagonal part of $\Gamma_n = \PBplus_{n+1}(\Z[\phi])$. 
Thus, a typical element of this quotient is 
the class of a diagonal matrix $\mathbf{d}=\textup{Diag}(\phi^{\epsilon_1}, \dots, \phi^{\epsilon_{n+1}})$ modulo scalar matrices. 
Hence, $\rho_i$ is given by 
\[\rho_i(\mathbf{d}): \mathbf{x}N_{i+1} \mapsto \mathbf{dxd}^{-1}N_{i+1}.\] 
By \cref{lem:gammas}, an element $\mathbf{x}N_{i+1} \in N_i/N_{i+1}$ has the form 
\[\mathbf{x}=\prod_{k=1}^{n-i}\ekl{k,k+i}(r_{k,k+i})N_{i+1},\] 
for some $r_{m,\ell} \in \Z[\phi]$, where the $e_{i,j}(r) \in \mbU_{n+1}(\Z[\phi])$ denote elementary matrices. One easily checks that the following identities hold in $B_{n+1}^+(\Z[\phi])$.  
\[ \textup{Diag}(u_1, \dots, u_{n+1}) \eij(r) (\textup{Diag}(u_1, \dots, u_{n+1}))^{-1} = \eij(u_i u_j^{-1} r).\] 
Thus, projecting onto $\Gamma_n = \PBplus_{n+1}(\Z[\phi])$ we see that 
\[\mathbf{dxd}^{-1}N_{i+1} = \prod_{k=1}^{n-i}\mathbf{d}\ekl{k,k+i}(r_{k,k+i})\mathbf{d}^{-1}N_{i+1} = \prod_{k=1}^{n-i}\ekl{k,k+i}(\phi^{\epsilon_k-\epsilon_{k+i}}r_{k,k+i})N_{i+1}.\] 
As $r_{k,k+i} \in \Z[\phi] = \Z \oplus \Z\phi \cong \Z^2$, we can write $r_{k,k+i}=a_{k,k+i}+\phi b_{k,k+i}$ for some $a_{k,k+i}, b_{k,k+i} \in \Z$. 
We then see that the action of $\rho_i(\mathbf{d})$ restricted on each matrix entry $(k,k+1)$ is given by 
\[\rho_i(\mathbf{d}): \ekl{k,k+i}(a_{k,k+i}+b_{k,k+i}) \mapsto \ekl{k,k+i}(\phi^{\epsilon_k-\epsilon_{k+i}}a_{k,k+i}+\phi^{\epsilon_k-\epsilon_{k+i}+1}b_{k,k+i}).\] 
This can be seen as the map $\Z[\phi] \to \Z[\phi]$ given by multiplication by~$\phi^{\epsilon_k-\epsilon_{k+i}}$. 

It then suffices to show that, for arbitrary $m \in \Z \setminus \{0\}$, the map $f_m:\Z[\phi] \to \Z[\phi]$ given by multiplication by $\phi^m$ has no nontrivial roots of unity as eigenvalues (when regarded as an element of $\GL_2(\Z)$).  
Because $\phi^2=\phi+1$ we see that, for any $a,b \in \Z$, 
\[f_1(a+\phi b)= 
b +(a+b)\phi.\] 
Thus, $f_1$ is represented by the matrix 
\[M_{f_1} = \begin{pmatrix}
	0 & 1 \\ 1 & 1
\end{pmatrix} \in \GL_2(\Z).\] 
It is clear that $f_m = f_1^m$, so that $f_m$ is represented my $M_{f_1}^m$.

Now, we see that the eigenvalues of $M_{f_1}$ are precisely $\phi$ and $\nu=\frac{1-\sqrt{5}}{2}$, i.e., the other root of the polynomial $x^2-x-1 \in \Z[x]$ besides $\phi$. This shows that the eigenvalues of $f_m$ are $\phi^m$ and $\nu^m$. Since $\vert\nu\vert < 1 < \phi$, no term of the sequences $(\nu^m)_{m\in \N}$, $(\nu^{-m})_{m\in \N}$, $(\phi^m)_{m\in \N}$, and $(\phi^{-m})_{m\in \N}$ approaches $1$. Thus, no root of unity is an eigenvalue of any $f_m$. 
\end{proof}

\noindent {\bf Finishing off the proof of \cref{thm:geometricapplication}.} Combine \cref{pps:GammasNR} and \cref{thm:KMcDVdB} for the manifolds $M_n$ with $\pi_1(M_n) \cong \Gamma_n$. \qed

\begin{rmk}
The proof of \cref{thm:geometricapplication} given above works the same for the groups $B^+_{n+1}(\Z[\phi])$. Denoting by $M_n'$ the corresponding solvmanifolds, their dimensions are $\dim(M_n') = \dim(M_n)+1 = (n+1)^2$. 

There are also other rings of integers $\ri_{\Q(\sqrt{d})}$ of quadratic number fields to which the techniques above apply; cf. \cref{pps:zft}. It would be interesting to classify solvmanifolds with the properties listed in \cref{thm:geometricapplication} that can be so constructed, and how they are distinguished geometrically.
\end{rmk}

%%%
\section{Reidemeister numbers, group extensions, triangular matrices} \label{sec:LemmataRinfty}
We embark the second part of the paper, first recalling standard results regarding twisted conjugacy classes and invariant subgroups and quotients. 
If $G$ fits into a short exact sequence $N \into G \onto Q$ with $N$ invariant under $\phee \in \Aut(G)$, we denote by $\phee' \in \Aut(N)$ and $\overline{\phee} \in \Aut(Q)$ the automorphisms induced by $\phee$. Explicitly, $\phee' = \phee|_N$ while $\barra{\phee}$ is given by $\barra{\phee}(gN) := \phee(g)N$. 

\begin{lem}[{\cite[Proposition~1.2]{Daciberg}}] \label{lem:desempre}
Suppose $\phee \in \Aut(G)$ and that $G$ fits into a short exact sequence $N \into G \onto Q$ where $N$ is $\phee$-invariant. Then 
$R(\phee) \geq R(\overline{\phee})$. 
\end{lem}

The following result has been used in the literature before. Since we are unaware of references containing a full proof and because similar arguments will be needed again later, we include a proof here for completeness. 

\begin{pps}\label{pps:phee} 
Suppose the short exact sequence $N \into G \onto Q$ splits and that $N$ is \emph{abelian} and $\phee$-invariant for $\phee \in \Aut(G)$. 
Then there exists an automorphism $\phee_0 \in \Aut(G)$ such that 
\begin{enumerate}
\item $N$ is $\phee_0$-invariant and $\phee_0' := \phee_0|_N = \phee|_N$,
\item the subgroup $Q \cong G/N$ of the semi-direct product $G = N \rtimes Q$ is also $\phee_0$-invariant, and both $\phee$ and $\phee_0$ induce the same automorphism $\barra{\phee}_0 = \barra{\phee} \in \Aut(Q)$ given by $\barra{\phee}_0(gN) = \barra{\phee}(gN) = \phee(g)N$, and
\item $R(\phee) = R(\phee_0)$.
\end{enumerate}
\end{pps}

In other words, in the set-up of \cref{pps:phee} one can replace $\phee \in \Aut(N \rtimes Q)$ by some automorphism that preserves {both} $N$ and $Q$ while having the same Reidemeister number as $\phee$.

\begin{proof} Each element~$g$ of $G \cong N \rtimes Q$ is of the form $g=nq$ for unique $n \in N$ and $q \in Q$. Thus, for any $q \in Q$, there exist unique $n_q \in N$ and $q_\phee \in Q$ such that $\phee(q)=n_q q_\phee$. 
We define $\phee_0$ by setting 
\[\phee_0(n)=\phee(n) \text{ for all } n \in N \text{ and }\phee_0(q)=q_\phee \text{ for all } q \in Q.\]
Notice that, for $d,q \in Q$ 
we have that $\phee_0(q d)=q_\phee d_\phee$. 
Indeed, by definition of the symbols $n_x$ and $x_\phee$ for $x \in Q$ we have 
\[(n_{qd}) (qd)_\phee = \phee(qd) = n_q q_\phee n_d d_\phee = n_q q_\phee n_d q_\phee^{-1} q_\phee d_\phee = (n_q q_\phee n_d q_\phee^{-1}) (q_\phee d_\phee),\]
 which yields $(qd)_\phee = q_\phee d_\phee$ by uniqueness of expression.

It is clear that $\phee_0(Q) \subseteq Q$, $\phee_0|_{N}=\phee|_{N}$ and $\barra{\phee}_0=\overline{\phee}$. We now show that $\phee_0$ is a homomorphism. Given $g,h \in G$ with $g=ab$ and $h=cd$, where $a,c \in N$ and $b,d \in Q$, it holds
\begin{align*} \phee_0(g h) &=\phee_0(a b c b^{-1} b d)\\
														 &=\phee(a b c b^{-1}) b_\phee d_\phee\\ 
														 &=\phee(a) b_\phee b_{\phee}^{-1} \phee(b c b^{-1}) b_\phee d_\phee.
\end{align*}
Using the fact that $b_{\phee} = n_{b}^{-1} \phee(b)$ we see that 
\[b_{\phee}^{-1} \phee(b c b^{-1}) b_\phee = \phee(b)^{-1} n_b \phee(b c b^{-1}) n_{b}^{-1} \phee(b).\]
Here we use the hypothesis that~$N$ is abelian, so that 
\[n_b \phee(b c b^{-1}) n_{b}^{-1} = \phee(b c b^{-1}), \quad \text{ hence } \quad b_{\phee}^{-1} \phee(b c b^{-1}) b_\phee = \phee(c).\] 
Consequently, 
\[ \phee_0(g h) =\phee(a) b_\phee \phee(c) d_\phee = \phee_0(g) \phee_0(h).\]

The injectivity of $\phee_0$ follows immediately from the injectivity of $\phee|_N$ and of $\phee_0|_Q$. The former is evident since $\phee_0|_N = \phee|_N \in \Aut(N)$ and the latter holds because $q_\phee = 1$ if and only if $q = 1$. 

To check that $\phee_0$ is surjective, let $g = nq \in G$ be arbitrary, where $n \in N$ and $q \in Q$ are uniquely determined by $g$. The surjectivity of $\phee$ implies that there exist $o \in N$ and $r \in Q$ such that $\phee(or) = q$. On the other hand, we may write $\phee(or) = \phee(o) n_r r_\phee$ with $n_r \in N$ and $r_\phee \in Q$ defined as in the beginning of the proof. Now, since $\phee|_N \in \Aut(N)$ there exist $m,l \in N$ such that $\phee(m) = n$ and $\phee(l) = n_r$. It then follows that
\[ g = nq = \phee(m) \phee(or) = \phee(m) \phee(o) n_r r_\phee = \phee(m) \phee(o) \phee(l) r_\phee = \phee_0(molr), \]
i.e. $\phee_0$ is surjective. Thus $\phee_0$ is in fact an automorphism. 

The equality $R(\phee) = R(\phee_0)$ can be extracted from observations due to Gon\c{c}alves in \cite[Section~1]{Daciberg}. Alternatively, it is a straightforward consequence of \cite[Theorem~3.2]{KarelSam}, as we now explain. In their notation, $\mf{R}(f,id_H)$ stands for the set of Reidemeister classes of an automorphism $f$ of a group $H$ --- in particular, $R(f) = | \mf{R}(f,id_H) |$. The cited theorem together with the properties $\overline{\phee} = \overline{\phee}_0$ and $\phee|_N = \phee_0|_N$ of $\phee_0 \in \Aut(G)$ then yield the following equalities of sets, where $\sqcup$ denotes the disjoint union.
\begin{align*} 
\mf{R}(\phee,id_G) & =  \bigsqcup_{[\pi(g)]_{\overline{\phee}} \in \mf{R}(\overline{\phee},id_Q)} (\mu_g \circ \hat{\text{\emph{\i}}}_g) (\mf{R}(\iota_g \circ \phee|_N, id_N)) \\
& =
 \bigsqcup_{[\pi(g)]_{\overline{\phee}_0} \in \mf{R}(\overline{\phee}_0,id_Q)} (\mu_g \circ \hat{\text{\emph{\i}}}_g) (\mf{R}(\iota_g \circ \phee_0|_N, id_N)) = \mf{R}(\phee_0,id_G). %\\
\end{align*}
(There is a slight abuse of notation here. The maps $\mu_g$ and $\hat{\text{\emph{\i}}}_g$ denote, on the first line, the maps $\mf{R}(\iota_g \circ \phee, id_G) \xrightarrow{\mu_g} \mf{R}(\phee, id_G)$ and $\mf{R}(\iota_g \circ \phee|_N, id_N) \xrightarrow{\hat{\text{\emph{\i}}}_g} \mf{R}(\iota_g \circ \phee, id_G)$, respectively, whereas on the second line they actually denote the maps $\mf{R}(\iota_g \circ \phee_0, id_G) \xrightarrow{\mu_g} \mf{R}(\phee_0, id_G)$ and $\mf{R}(\iota_g \circ \phee_0|_N, id_N) \xrightarrow{\hat{\text{\emph{\i}}}_g} \mf{R}(\iota_g \circ \phee_0, id_G)$, respectively.) 
The above thus shows that the Reidemeister number of $\phee$ (and also of $\phee_0$, by definition) only depend on the Reidemeister numbers of their restrictions to $N$ and $Q$, whence $R(\phee) = R(\phee_0)$.
\end{proof}

\begin{rmk} 
\cref{pps:phee} is used in~\cite{TabackWong0}, but the hypothesis that~$A$ is $\phee$-invariant and abelian is missing in \cite[Corollary~2.3]{TabackWong0}. This omission does not affect the proof of their main result \cite[Theorem~5.4]{TabackWong0} since their corresponding normal subgroup~$A$ is characteristic and abelian. 
\end{rmk}

Since we shall deal with soluble groups later on, we will frequently need auxiliary results on Reidemeister classes of abelian groups. 

\begin{lem} \label{lem:ReidemeisterAbelian}
Suppose $A$ is a {finitely generated} abelian group and let $\phee \in \Aut(A)$. Denote by $t(A)$ the torsion subgroup of $A$ and let $r = \rk(A/t(A))$. The following are equivalent.
\begin{enumerate}
\item $R(\phee) = \infty$.
\item The automorphism $\phee$ has {infinitely many} fixed points.
\item $R({\phee}_{\mathrm{tf}}) = \infty$, where ${\phee}_{\mathrm{tf}}$ is the automorphism induced by $\phee$ on the torsion-free part $A / t(A) \cong \Z^r$.
\item The above map ${\phee}_{\mathrm{tf}} \in \Aut(A/t(A)) \cong \GL_r(\Z)$ has $1$ as Eigenvalue.
\item The above map ${\phee}_{\mathrm{tf}} \in \Aut(A/t(A))$ admits some nonzero fixed point.
\end{enumerate}
\end{lem}

The statements above make sense because the torsion part of an abelian group is (fully) characteristic. 
\cref{lem:ReidemeisterAbelian} has been extensively used in the literature; see e.g. \cite[Section~2 and Lemma~4.1]{Romankov}, \cite{KarelDacibergAbelian} or \cite[Section~4]{KarelSam}.

\subsection{Notation for groups of triangular matrices} \label{sec:osgrupos} 

Throughout we let $R$ denote a {commutative} ring {with unity}, unless explicitly said otherwise. We let $\Addi(R) = (R,+)$ and $\Mult(R) = (R^\times, \cdot)$ denote the underlying additive group and the group of units of the ring $R$, respectively. Denote by $\GL_n(R)$ the $n \times n$ general linear group over $R$. 

Here we typically denote affine $\Z$-subschemes of $\GL_n$ with boldface letters. 
The (standard) Borel subgroup of $\GL_n(R)$ is its soluble subgroup $\mbB_n(R) \leq \GL_n(R)$ of upper triangular matrices. Consider its subgroups $\mbU_n(R) \leq \mbB_n(R)$ and $\mbD_n(R) \leq \mbB_n(R)$ of (upper) unitriangular and diagonal matrices, respectively. We recall some fact about these groups that will be used throughout.

Write $\eij(r) \in \mbU_n(R)$ for the elementary matrix whose diagonal entries are~$1$, its $(i,j)$-entry with $i < j$ is $r \in R$, and all other entries are zero. 
For $i \in \set{1,\ldots,n}$ and $u \in R^\times$, denote by $\di(u) \in \mbD_n(R)$ the diagonal matrix whose $i$-th entry is $u$ and all other (diagonal) entries are~$1$. 
For a fixed $i$ and a subgroup $S \leq R^\times$, we write $\mc{D}_i(S) := \gera{ d_i(u) \mid u \in S}$. 

The group-theoretic relationship between elementary and diagonal matrices is most easily described via the following well-known sets of relations \cite{HahnO'Meara, Silvester}, which have appeared in the proofs of \cref{lem:gammazero} and \cref{pps:GammasNR}. Throughout this paper we write group commutators as $[g,h] = ghg^{-1}h^{-1}$.
\begin{align}
\label{rel:commutators}
\begin{split}
[\eij(r),\ekl{k,l}(s)] & =
\begin{cases}
\ekl{i,l}(rs) & \mbox{if } j=k,\\
1 & \mbox{if } i \neq l \text{ and } k \neq j,
\end{cases} \\
\di(u) \ekl{k,l}(r) \di(u)^{-1} & = 
\begin{cases}
\ekl{k,l}(ur) & \mbox{if } i=k,\\
\ekl{k,l}(u^{-1}r) & \mbox{if } i=l, \\
1 & \mbox{otherwise}.
\end{cases}
\end{split}
\end{align}
In particular, if $d = d_1(u_1) \cdots d_n(u_n) \in \mbD_n(R)$, one has 
\[ d \eij(r) d^{-1} = \eij(u_i u_j^{-1} r). \] 
The set of relations~\eqref{rel:commutators} is also referred to as \emph{elementary} or \emph{commutator relations}. Elements of $\mbU_n(R)$ thus decompose uniquely as a product of the matrices $\eij(r)$, ordered according to the superdiagonals. That is, if $x \in \mbU_n(R)$, then there are uniquely determined elements $r_{i,j} \in R$ for which 
\begin{align} \label{rel:elementsofUn} 
\begin{split}
x = & \phantom{.} e_{1,2}(r_{1,2}) e_{2,3}(r_{2,3}) \cdot \ldots \cdot e_{n-1,n}(r_{n-1,n}) \cdot \\
 & \cdot  e_{1,3}(r_{1,3}) \cdot \ldots \cdot e_{n-2,n}(r_{n-2,n}) \cdot \ldots \ldots \cdot e_{1,n}(r_{1,n}). 
\end{split}
\end{align} 

It is clear that $\mbB_n(R)$ splits as a semi-direct product $\mbB_n(R) = \mbU_n(R) \cdot \mbD_n(R) \cong \mbU_n(R) \rtimes \mbD_n(R)$, where the action of $\mbD_n(R)$ on $\mbU_n(R)$ is given by conjugation~\eqref{rel:commutators}. 
As we assume $R$ to be commutative, the diagonal subgroup $\mbD_n(R)\cong \Mult(R)^n$ is abelian and the unitriangular subgroup $\mbU_n(R)$ is nilpotent of class $n-1$. In particular, $\mbB_n(R)$ is soluble of class at most $n$. 

We now introduce a variant of $\mbB_n(R)$ whose diagonal subgroup is torsion-free.
Recall that any abelian group $A$ decomposes as the direct product $A = t(A) \cdot A_{\mathrm{tf}} \cong t(A) \times A_{\mathrm{tf}}$ of its torsion subgroup $t(A)$ and a torsion-free part $A_{\mathrm{tf}}$. More precisely, the torsion subgroup 
\[ t(A) = \set{ a \in A \mid a \text{ has finite order} } \]
is a (fully) characteristic subgroup and admits a group-theoretic complement, denoted $A_{\mathrm{tf}} \leq A$. 
While $A_{\mathrm{tf}}$ is not unique \emph{as a subset} of $A$, it is unique up to group isomorphism. 

Given a subgroup $S \leq R^\times$ write $D_n(S) := \gera{\set{d_i(u) \in \mbD_n(R) \mid u \in S}} = \prod_{i=1}^n \mc{D}_i(S)$ for the corresponding subgroup of diagonal matrices. The group $B_n^+(R)$ is then defined (intrinsically) as 
\[ B_n^+(R) :=  \mbU_n(R) \rtimes \left(\frac{\mbD_n(R)}{D_n(t(R^\times))}\right),\] 
where the action of ${\mbD_n(R)}/{D_n(t(R^\times))}$ on $\mbU_n(R)$ is induced by relations~\eqref{rel:commutators}. Extrinsically, writing $R_{\mathrm{tf}}$ for any choice of complement of $t(R^\times)$, one has $B_n^+(R) \cong \mbU_n(R) \rtimes D_n(R_{\mathrm{tf}})$ as an explicit subgroup of $\mbB_n(R)$. Thus $B_n^+(R)$ is just `the' subgroup of $\mbB_n(R)$ without torsion on the diagonal, which is unique up to isomorphism. Of course, if $R^\times$ is itself a torsion group, then the diagonal part of $B_n^+(R)$ is trivial, in which case $B_n^+(R) = \mbU_n(R)$.

We now turn to a further group of interest which is obtained by modding out scalar matrices. To be precise, let $Z_n(R)$ be the central subgroup 
\[Z_n(R) = \set{u\cdot \mb{1}_n \in \GL_n(R) \mid u \in R^\times} = \set{d_1(u)\cdots d_n(u) \mid u \in R^\times} \leq \mbD_n(R),\]
where we denote by $\mb{1}_n$ the $n \times n$ identity matrix. The \emph{projective upper triangular groups} $\PB_n(R)$ and $\PBplus_n(R)$ are defined as  
\[ \PB_n(R) = \frac{\mbB_n(R)}{Z_n(R)} = \mbU_n(R) \rtimes \frac{\mbD_n(R)}{Z_n(R)}\] 
and 
\[\PBplus_n(R) = \mbU_n(R) \rtimes \frac{\mbD_n(R)}{D_n(t(R)) Z_n(R)} \cong \mbU_n(R) \rtimes \frac{D_n(R_{\mathrm{tf}})}{Z_n(R_{\mathrm{tf}})}, \]
where $Z_n(R_{\mathrm{tf}}) = \{ u \cdot \mb{1}_n \in \GL_n(R) \mid u \in R_{\mathrm{tf}} \} = D_n(R_{\mathrm{tf}}) \cap Z_n(R)$.

\begin{rmk} \label{obs:centerBn} If $R$ is an integral domain, then $Z_n(R)$ (resp. $Z_n(R_{\mathrm{tf}})$) is either trivial or coincides with the center of $\mbB_n(R)$ (resp. of $B_n^+(R)$). Thus $\PB_n(R)$ (resp. $\PBplus_n(R)$) is a quotient of $\mbB_n(R)$ (resp. of $B_n^+(R)$) modulo a \emph{characteristic} subgroup whenever $R$ is an integral domain. 
\end{rmk}

The elements of the diagonal subgroups of $\PB_n(R)$ and $\PBplus_n(R)$ will typically be denoted by $[d]$. 
Any element $g \in \PB_n(R)$ (resp. $\PBplus_n(R)$) can thus be written uniquely as $g = u[d]$ with $u \in \mbU_n(R)$ and $d \in \mbD_n(R)$ (resp. $d \in \mbD_n(R_{\mathrm{tf}})$). Moreover, relations~\eqref{rel:commutators} imply that 
\begin{align} \label{rel:projectiveconjugation}
[d] \ekl{k,l}(r) [d]^{-1} = \ekl{k,l}(u_k u_l^{-1} r) & \phantom{a} & \text{ for } d = d_1(u_1) \cdots d_n(u_n).
\end{align}

The group $\mbB_2(R)$ has a close relative, namely the \emph{group of affine transformations} of the base ring $R$, defined as 
\[
\Aff(R) = \AffR{R} \cong \begin{pmatrix} * & * \\ 0 & 1 \end{pmatrix} \leq \GL_2(R).
\] 
In the above, the units $\Mult(R) = (R^\times,\cdot)$ act on the underlying additive group $\Addi(R) = (R,+)$ by multiplication. One defines $\Aff^+(R)$ similarly to $B_2^+(R)$, and it holds 
\[
\Aff^+(R) \cong \left\{\begin{pmatrix} \diamond & * \\ 0 & 1 \end{pmatrix} \, \middle| \, \diamond \in R_{\mathrm{tf}}, \ast \in R \right\} \leq \GL_2(R).
\] 
As it turns out, $\Aff(R)$ and $\Aff^+(R)$ are not new. 

\begin{lem} \label{lem:AffisPB}
$\PB_2(R) \cong \Aff(R)$ and $\PBplus_2(R) \cong \Aff^+(R)$. In particular, if $R$ is an integral domain, then such groups are either centerless or abelian.
\end{lem}

\begin{proof}
We argue for $\PB_2(R)$ and $\Aff(R)$ --- the other case is entirely analogous. Any $g \in \PB_n(R)$ can be written (uniquely) as $ g= e_{1,2}(r)[d_1(u_1)d_2(u_2)]$. The assignment 
\[
\xymatrix@R=2mm{
f : \PB_2(R) \ar[r] & {\Aff(R) \cong \left(\begin{smallmatrix} * & * \\ 0 & 1 \end{smallmatrix}\right) \leq \GL_2(R)} \\
  {e_{1,2}(r)[d_1(u_1)d_2(u_2)]} \ar@{|->}[r] & 
	{e_{1,2}(r) d_1(u_1 u_2^{-1})}
}
\]
is easily checked to be an isomorphism. 

\cref{obs:centerBn} assures that $Z_2(R)$ is either trivial or coincides with the center $Z(\mbB_2(R))$ of $\mbB_2(R)$, assuming $R$ to be an integral domain. Now if $Z_2(R)$ is trivial, then $\PB_2(R)=\mbB_2(R)=\mbU_2(R)$ is abelian. Otherwise $Z_2(R)=Z(\mbB_2(R))$, so that $\PB_2(R)=\mbB_2(R)/Z(\mbB_2(R))$ has trivial center. (In fact, direct computations show that the only matrices in $\mbB_2(R)$ commuting with all elements are the scalar ones.) 
\end{proof}

\subsection{Unitriangular matrices as a characteristic subgroup} \label{sec:IsUncharacteristic}

The reader familiar with linear algebraic groups (over a field $\K$) might recall that $\mbU_n$ is the commutator subgroup scheme of $\mbB_n$, hence (algebraically) characteristic in $\mbB_n$ \cite{BorelLAG}. And, in the previous \cref{sec:LemmataRinfty}, we collected results about the computation of Reidemeister numbers using characteristic subgroups and quotients. This naturally raises the question of whether the normal subgroup $\mbU_n(R)$ is characteristic in the groups considered here. However, this is not true in the general abstract case of $\mbB_n(R)$ and $B_n^+(R)$, not even over integral domains, as we now show.

\begin{pps} \label{pps:UnNOTcharacteristic}
Let $R$ be an integral domain and $S \in \{R^\times, R^\times_{\mathrm{tf}}\}$. If there exists a nontrivial group homomorphism $\veps : (R,+) \to (S,\cdot)$ such that $\veps(ur)=\veps(r)$ for all $u \in S$, then $\mbU_2(R)$ is \emph{not} a characteristic subgroup of $\mbB_2(R)$ (in case $S = R^\times$) or of $B^+_2(R)$ (in case $S = R^\times_{\mathrm{tf}}$).
\end{pps}

\begin{proof}
For simplicity, write $G(R)$ for $\mbB_2(R)$ (case $S = R^\times$) or $B_2^+(R)$ (case $S = R^\times_{\mathrm{tf}}$). Such a map $\veps$ induces a map
\[
\xymatrix@R=2mm{
\phee_\veps : G(R) \ar[r] & G(R) \\
  {\begin{pmatrix} u & r \\ 0 & v \end{pmatrix}} \ar@{|->}[r] & 
	{\begin{pmatrix} \veps(r) & 0 \\ 0 & \veps(r) \end{pmatrix}} \cdot {\begin{pmatrix} u & r \\ 0 & v \end{pmatrix}}.
}
\]
Since $\veps$ is a homomorphism with $\veps(ur)=\veps(r)$ for all $u \in S$ and because the matrix $\left(\begin{smallmatrix} \veps(r) & 0 \\ 0 & \veps(r) \end{smallmatrix}\right)$ is central, one readily checks that $\phee_\veps$ is also a homomorphism. It is also straightforward that $\ker(\phee_\veps)$ is trivial. As $G(R) = \gera{\left(\begin{smallmatrix} u & 0 \\ 0 & v \end{smallmatrix}\right), \left(\begin{smallmatrix} 1 & r \\ 0 & 1 \end{smallmatrix}\right) \mid u,v \in S, r \in R}$, the map $\phee_\veps$ is also surjective since $\left(\begin{smallmatrix} u & 0 \\ 0 & v \end{smallmatrix}\right) = \phee_\veps\left(\left(\begin{smallmatrix} u & 0 \\ 0 & v \end{smallmatrix}\right)\right)$ and $\left(\begin{smallmatrix} 1 & r \\ 0 & 1 \end{smallmatrix}\right) = \phee_\veps\left(\left(\begin{smallmatrix} \veps(r)^{-1} & 0 \\ 0 & \veps(r)^{-1} \end{smallmatrix}\right) \cdot \left(\begin{smallmatrix} 1 & r \\ 0 & 1 \end{smallmatrix}\right)\right)$. 
It is clear that $\mbU_2(R)$ is not $\varphi_\veps$-invariant, hence the result.
\end{proof}

\begin{exm} \label{exm:UnNOTcharacteristic}
For $G = \mbB_2$, take $R = \Z[t]$. Given $r = \sum_{i=0}^N f_i t^i \in R$ and taking $\veps(r) = (-1)^f \in \{-1,1\} = R^\times$, where $f = \sum_{i=0}^N f_i \in \Z$, one can apply \cref{pps:UnNOTcharacteristic}. Similarly for $G = B_2^+$, take $R = \Z[t,t^{-1}]$ (and $R^\times_{\mathrm{tf}} = \{t^n \mid n \in \Z\}$) and define $\veps(r) = t^f$, where $r = \sum_{i=-\infty}^\infty f_i t^i$ and $f = \sum_{i=-\infty}^\infty f_i \in \Z$ --- here, all but finitely many coefficients $f_i$ are nonzero. 

Hence there do exist integral domains for which $\mbU_n(R)$ is not necessarily characteristic in $\mbB_n(R)$ or in $B_n^+(R)$.
\end{exm}

Nevertheless, eliminating scalar matrices makes $\mbU_n(R)$ characteristic. Recall that the \emph{Hirsch--Plotkin radical} of a group $G$ is the {unique} maximal subgroup of $G$ with respect to being normal {and} locally nilpotent. (If $P$ is a group-theoretic property, a group $G$ is locally $P$ if every finitely generated subgroup of $G$ has property $P$.)

\begin{pps}\label{pps:char} 
If $R$ is an integral domain, then $\mbU_n(R)$ is the Hirsch--Plotkin radical (hence a characteristic subgroup) of $\PB_n(R)$ and of $\PBplus_n(R)$.  
\end{pps}

While the proposition is surely well-known, we were unable to find explicit references for it. For completeness we give a rather elementary proof showcasing the use of relations from \cref{sec:osgrupos}. The familiar reader might recall a result of Gruenberg \cite[Theorem~4]{GruenbergHypercenter} stating that the Hirsch--Plotkin radical of finitely generated linear groups is always nilpotent. Particularly, \cref{pps:char} recovers this fact for $\PB_n(R)$ and $\PBplus_n(R)$ --- however, we do not assume these groups to be finitely generated in \cref{pps:char}.

\begin{proof} 
We prove the case of $\PB_n(R)$, that of $\PBplus(R)$ being entirely analogous. If $R^\times$ is trivial, then $\PB_n(R) = \mbU_n(R)$, and thus $\PB_n(R)$ is its own Hirsch--Plotkin radical since it is nilpotent. 
(In the `positive' case, if $R^\times$ has no torsion-free units, then $\PBplus_n(R) = \mbU_n(R)$.) We thus assume from now on that $|R^{\times}|\geq 2$. (For the `positive' case, $|R^\times \setminus t(R^\times)| \geq 1$.)

Since $\mbU_n(R)$ is normal and nilpotent, it is contained in the Hirsch--Plotkin radical. Now let $M$ be an arbitrary locally nilpotent normal subgroup of $\PB_n(R)$. We shall show that $M \subseteq \mbU_n(R)$, which then implies that the Hirsch--Plotkin radical agrees with $\mbU_n(R)$. 
Suppose, on the contrary, that there exists $m \in M \setminus \mbU_n(R)$. 
Let $v \in \mbU_n(R)$ and $d = d_1(u_1)\dots d_n(u_n)\in \mbD_n(R)$ be such that $m=v[d]$.  

The assumption $m \notin \mbU_n(R)$ means that $[d]\neq [\mb{1}_n]$ in $\mbD_n(R)/Z_n(R)$. Consequently, there is an index $k \in \{1, \dots, n-1\}$ for which $u_k\neq u_{k+1}$. 
By the description~\eqref{rel:elementsofUn} of elements of $\mbU_n(R)$ and by the commutator relations~\eqref{rel:commutators}, the element $v \in \mbU_n(R)$ can be expressed as a product $v=e_{k,k+1}(r_{k,k+1})x$, where 
\begin{equation}\label{unipelem} 
x= \prod_{(i,j) \neq (k,k+1)}e_{i,j}(r_{i,j})\in \mbU_n(R). 
\end{equation}
To ease notation, write 
\[\mathbf{r}_j=r_{j,j+1} \quad \text{ and } \quad \overline{\mathbf{e}}_j(a)=e_{j,j+1}(a) \quad \text{ for each }  j \in \{1,\ldots,n-1\}, a \in R.\]
This allows us to write 
\[m = \overline{\mathbf{e}}_k(\mathbf{r}_{k}) x [d].\]

We make a case distinction for $\mathbf{r}_k$. First assume $\mathbf{r}_{k}\neq 0$ and set $\delta=d_k(u_k)d_{k+1}(u_{k+1}) \in \mbD_n(R)$. Since $M$ is normal, we must have $[\delta] m [\delta]^{-1} \in M$. Moreover, 
\[ [\delta] m [\delta]^{-1} = [\delta] \overline{\mathbf{e}}_k(\mathbf{r}_{k})[\delta]^{-1}[\delta] x [d] [\delta]^{-1} = \overline{\mathbf{e}}_k(u_k u_{k+1}^{-1} \mathbf{r}_{k})  [\delta] x [\delta]^{-1}[d] \]
by relations~\eqref{rel:projectiveconjugation}. 
Consider $y= [\delta] x [\delta]^{-1}$. Similarly to \cref{unipelem}, we have
\[y=\prod_{(i,j) \neq (k,k+1)}e_{i,k}(\hat{r}_{i,j}) \in \mbU_n(R)\] 
for some $\hat{r}_{i,j} \in R$. 
We can then  write 
 $[\delta] m [\delta]^{-1}=\overline{\textbf{e}}_{k}(u_ku_{k+1}^{-1}\mathbf{r}_{k})y[d]$. 
Let $s$ be the commutator of $m$ and $[\delta]$, i.e., $s = [m,[\delta]] \in M$. One has 
\begin{align*} 
s & = \overline{\mathbf{e}}_{k}(\mathbf{r}_{k})x[d]\left(\overline{\mathbf{e}}_{k}(u_ku_{k+1}^{-1}\mathbf{r}_{k})y[d]\right)^{-1} \\
									& = \overline{\mathbf{e}}_{k}(\mathbf{r}_{k})xy^{-1}\overline{\mathbf{e}}_{k}(-u_ku_{k+1}^{-1}\mathbf{r}_{k}) = \overline{\mathbf{e}}_{k}(\mathbf{r}_{k}(1-u_ku_{k+1}^{-1}))w
\end{align*} 
for some $w \in \mbU_n(R)$ of the form $w = \prod_{(i,j)\neq(k,k+1)} e_{i,j}(\hat{\hat{r}}_{i,j})$. 
Using relations~\eqref{rel:commutators} and~\eqref{rel:projectiveconjugation} again, the commutator $[s,m] \in M$ is given by 
\begin{align*} 
[s,m] 	&= \overline{\mathbf{e}}_{k}(\mathbf{r}_{k}(1-u_ku_{k+1}^{-1}))w \overline{\mathbf{e}}_{k}(\mathbf{r}_{k})x[d] s^{-1}m^{-1}\\
											&= \overline{\mathbf{e}}_{k}(\mathbf{r}_{k}(1-u_ku_{k+1}^{-1})+\mathbf{r}_{k}) \widetilde{w} [d] s^{-1}m^{-1},
\end{align*}
for some $\widetilde{w} = \prod_{(i,j)\neq(k,k+1)} e_{i,j}(\widetilde{r}_{i,j}) \in \mbU_n(R)$, and 
\begin{align*} 
[d]s^{-1}m^{-1}	&= [d]w^{-1} \overline{\mathbf{e}}_{k}(-\mathbf{r}_{k}(1-u_ku_{k+1}^{-1})) [d]^{-1}x^{-1}\overline{\mathbf{e}}_{k}(-\mathbf{r}_{k})\\
															&= [d]w^{-1}[d]^{-1} \overline{\mathbf{e}}_{k}(-u_ku_{k+1}^{-1}\mathbf{r}_{k}(1-u_ku_{k+1}^{-1}))x^{-1}\overline{\mathbf{e}}_{k}(-\mathbf{r}_{k})\\
															&= \overline{\mathbf{e}}_{k}(-u_ku_{k+1}^{-1}\mathbf{r}_{k}(1-u_ku_{k+1}^{-1})-\mathbf{r}_{k})\widetilde{x}
\end{align*}
for some $\widetilde{x} = \prod_{(i,j)\neq(k,k+1)} e_{i,j}(\widetilde{\widetilde{r}}_{i,j}) \in \mbU_n(R)$. Substituting, 
\begin{align*} 
[s,m] 	&= \overline{\mathbf{e}}_{k}(\mathbf{r}_{k}(1-u_ku_{k+1}^{-1})+\mathbf{r}_{k})\widetilde{w} \overline{\mathbf{e}}_{k}(-u_ku_{k+1}^{-1}\mathbf{r}_{k}(1-u_ku_{k+1}^{-1})-\mathbf{r}_{k})\widetilde{x}\\ 
 &= \mathbf{e}_{k}(\mathbf{r}_{k}(1-u_ku_{k+1}^{-1})^2) w'_1
\end{align*} 
for some $w'_1 = \prod_{(i,j)\neq(k,k+1)} e_{i,j}(t_{i,j,1}) \in \mbU_n(R)$. Inductively, one shows that for every $\ell \in \N$ there is a $w'_\ell = \prod_{(i,j)\neq(k,k+1)} e_{i,j}(t_{i,j,\ell}) \in \mbU_n(R)$ such that 
\begin{equation}\label{nonnilp} [s,_\ell m]= \overline{\mathbf{e}}_{k}(\mathbf{r}_{k}(1-u_ku_{k+1}^{-1})^\ell) w'_\ell,\end{equation}
where $[s,_\ell m]$ denotes the iterated commutator $[[s, m], \ldots,m] \in M$ with $\ell$ occurrences of $m$. 
Since $R$ is an integral domain and $u_k \neq u_{k+1}$, it follows that $\mathbf{r}_{k}(1-u_ku_{k+1}^{-1})^\ell\neq 0$ for all $\ell \in \N$ and thus $[s,_\ell m] \neq 1$. This contradicts the hypothesis that the (finitely generated) subgroup $\langle s, m\rangle$ of $M$ is nilpotent. Therefore $m \in \mbU_n(R)$. 

Now assume $\mathbf{r}_{k}=0$, that is, that $m=x[d] \in M$ where $x$ is given as in \cref{unipelem}. Since $M$ is normal, we have $\widetilde{m}:=\overline{\mathbf{e}}_{k}(1)m\overline{\mathbf{e}}_{k}(1)^{-1} \in M$. 
Moreover, 
\begin{align*} \widetilde{m} & :=\overline{\mathbf{e}}_{k}(1)m\overline{\mathbf{e}}_{k}(1)^{-1} 
																	=  \overline{\mathbf{e}}_{k}(1)x[d]\overline{\mathbf{e}}_{k}(-1)[d]^{-1}[d]\\
																	&=  \overline{\mathbf{e}}_{k}(1)x\overline{\mathbf{e}}_{k}(-u_ku_{k+1}^{-1})[d] =  \overline{\mathbf{e}}_{k}(1-u_ku_{k+1}^{-1})z[d]
\end{align*} 
for some $z = \prod_{(i,j)\neq(k,k+1)} e_{i,j}(t'_{i,j}) \in \mbU_n(R)$. In particular, such a $z$ is an element of $M$ that fits the computations done in the previous case $\mathbf{r}_{k}\neq 0$. Thus, proceeding as before, we can find $\widetilde{s} \in M$ such that $[\widetilde{s},_\ell \widetilde{m}]\neq 1$ for all $\ell \in \N$, contradicting nilpotency of the (finitely generated) subgroup $\langle \widetilde{s}, \widetilde{m} \rangle \leq M$. Hence $M \subseteq \mbU_n(R)$, as claimed. 
\end{proof}

Of course, there are other cases where $\mbU_n(R)$ is characteristic in $\mbB_n(R)$, as the following folklore result states, generalizing part of \cref{lem:gammazero}. 

\begin{lem} \label{pps:caracteristico} Consider the following conditions: 
\begin{enumerate} 
	\item\label{char2} There exists $u \in R^{\times}$ such that $u-1\in R^{\times}$; 
	\item\label{char3} Each $r \in R$ is of the form $r=u+v$ for some $u,v \in R^{\times}$;
	\item\label{char2PLUS} There exists a complement $R^\times_{\mathrm{tf}} \subseteq R^\times$ for the torsion subgroup $t(R^\times) \subseteq R^\times$ and a unit $u \in R^\times_{\mathrm{tf}}$ such that $u-1 \in R^\times_{\mathrm{tf}}$; 
	\item\label{char3PLUS} There exists a complement $R^\times_{\mathrm{tf}}$ for the torsion subgroup $t(R^\times) \subseteq R^\times$ such that each $r \in R$ is of the form $r=u-v$ for some $u,v \in R^\times_{\mathrm{tf}}$.
\end{enumerate}
If~\eqref{char2} or~\eqref{char3} holds, then $\mbU_n(R)$ is the commutator subgroup of $\PB_n(R)$ \emph{and of} $\mbB_n(R)$. Similarly, if~\eqref{char2PLUS} or~\eqref{char3PLUS} holds, then $\mbU_n(R)$ is the commutator subgroup of \emph{any} $G(R)$ with $G \in \{\PB_n, \mbB_n, \PBplus_n, B_n^+\}$. 
\end{lem}
\begin{proof} 

We first argue for $\mbB_n(R)$ --- the case of $\PB_n(R)$ is analogous. 
The inclusion $[\mbB_n(R),\mbB_n(R)] \subseteq \mbU_n(R)$ is clear by relations~\eqref{rel:commutators}. (In particular, $[B_n^+(R),B^+_n(R)] \subseteq \mbU_n(R)$ holds as well.) 
If~\eqref{char2} holds, one obtains the reverse inclusion by observing that, for $1 \leq i < j \leq n$ and $s \in R$, one has 
\[e_{i,j}(s)=[d_i(u),e_{i,j}((u-1)^{-1}s)].\] 
If~\eqref{char3} holds, the reverse inclusion 
follows from the following identity:
\[[d_i(u), e_{i,j}(1)]\cdot [d_i(-v), e_{i,j}(1)^{-1}]=e_{i,j}(u-1)e_{i,j}(v+1) =e_{i,j}(u+v),\]
for $1 \leq i < j \leq n$ and $u,v \in R^{\times}$.

Now consider the case of $B_n^+(R) \cong \mbU_n(R) \rtimes D_n(R^\times_{\mathrm{tf}})$ --- again, $\PBplus_n(R)$ is dealt with similarly. If~\eqref{char2PLUS} holds, the inclusion $\mbU_n(R) \subseteq [B_n^+(R),B^+_n(R)]$ is obtained as for $\mbB_n(R)$. Assuming~\eqref{char3PLUS}, use the identity 
\[e_{i,j}(r) = e_{i,j}(u-v) = e_{i,j}(u-1)e_{i,j}(-v+1) = [d_i(u), e_{i,j}(1)]\cdot [d_i(v), e_{i,j}(1)^{-1}],\]
which concludes the proof.
\end{proof}

We remark that, in \cref{pps:caracteristico}, condition~\eqref{char3PLUS} is the strongest. Indeed, \eqref{char3PLUS} implies \eqref{char3} since $-v$ with $v\in R^\times_{\mathrm{tf}}$ is also a unit. Moreover, condition~\eqref{char3} implies condition~\eqref{char2}, and similarly \eqref{char3PLUS} implies \eqref{char2PLUS} --- simply take $r=1$. 
Despite looking rather harmless, conditions~\eqref{char3} and~\eqref{char3PLUS} touch upon difficult questions from ring theory. To be precise, the problem of classifying rings whose elements are sums of units is wide open. Major breakthroughs in the topic are results by Raphael~\cite{Raphael} and Henriksen~\cite{HenriksenGenByUnits} in the 1970s. 
One could hope that the problem is completely solved for certain classes of rings or domains with more structure. But even for $S$-arithmetic rings the question is still being actively investigated. We refer the reader, e.g., to \cite{FreiUnitsOK,JardenNarkiewicz} for recent number-theoretic progress in this direction.

\section{Theorem~\ref{thm:Aff} and its proof} \label{sec:provadothmA}

\cref{thm:Aff} gives a sufficient criterion for our groups of interest to have property~$\Ri$ by reducing the problem to quotients of the groups $\Aff(R)$ and $\Aff^+(R)$. We restate it below in a more precise form.

\theoremstyle{plain}
\newtheorem*{thmA}{Theorem~\ref{thm:Aff}}
\begin{thmA} 
Let $R$ be an integral domain. Then the following hold. 
\begin{enumerate}
    \item\label{thmA1} Given $n \geq 2$ and $G \in \{\mbB_n, \PB_n\}$ (resp. $G^+ \in \{B_n^+, \PBplus_n\}$) and an arbitrary automorphism $\phee \in \Aut(G(R))$ (resp. $\phee \in \Aut(G^+(R))$), there exists an automorphism $\psi \in \Aut(\Aff(R))$ (resp. $\psi \in \Aut(\Aff^+(R))$) induced by $\phee$ and such that $R(\phee) \geq R(\psi_{\mathrm{tf}})$. Here, $\psi_{\mathrm{tf}}$ denotes the automorphism induced by $\psi$ on the torsion-free part of the diagonal subgroup $\Aff(R)/\mbU_2(R)$ (resp. $\Aff^+(R)/\mbU_2(R)$).
    \item\label{thmA2} Assume $R^\times$ is finitely generated. If $\psi_{\mathrm{tf}}$ as above fixes infinitely many points, then the original automorphism $\phee$ of $G(R)$ (resp. of $G^+(R)$) has $R(\phee) = \infty$. In particular, if each \emph{arbitrary} automorphism $\psi$ of $\Aff(R)$ (resp. of $\Aff^+(R)$) is such that $|\mathrm{Fix}(\psi_{\mathrm{tf}})|=\infty$, then all the groups $G(R)$ with $G \in \{\mbB_n, \PB_n \mid n \geq 2\}$ (resp. $G^+(R)$ with $G \in \{B_n^+, \PBplus_n \mid n \geq 2\}$) have property~$\Ri$.
\end{enumerate} 
\end{thmA}

The statement of \cref{thm:Aff} from the introduction is recovered as follows: the first part of the statement is a special case of part~\eqref{thmA1} above since $\Aff(R) \cong \PB_2(R)$ (resp. $\Aff^+(R) \cong \PBplus_2(R)$) --- cf. \cref{lem:AffisPB}; the remainder follows from part~\eqref{thmA2} combined with the fact that the torsion-free parts of $\Aff(R)/\mbU_2(R)$ and of $\Aff^+(R)/\mbU_2(R)$ agree --- they are isomorphic to $R_{\mathrm{tf}}^\times$; cf. \cref{sec:osgrupos}. 

The core of \cref{thm:Aff} lies in part~\eqref{thmA1}, which easily yields part~\eqref{thmA2}.

\begin{proofof}{Theorem~\ref{thm:Aff}\eqref{thmA2}}
Under the hypotheses of \eqref{thmA2}, \cref{lem:ReidemeisterAbelian} yields that $R(\psi_{\mathrm{tf}}) = \infty$ if and only if $\psi_{\mathrm{tf}}$ has infinitely many fixed points, whence $R(\phee) = \infty$ for any $\phee \in \Aut(G(R))$ (resp. $\Aut(G^+(R))$) in case \eqref{thmA1} holds.
\end{proofof}

The remainder of this section is thus dedicated to the proof of \cref{thm:Aff}\eqref{thmA1}. The main idea can be outlined as follows. As elucidated in \cref{sec:LemmataRinfty} and in the literature, characteristic subgroups help in estimating Reidemeister numbers. In contrast, here we identify a special subgroup $\mb{W}_n(R)$ of $\PB_n(R)$ (resp. $W_n(R) \leq\PBplus_n(R)$)  that is {not} necessarily characteristic but still retains some information about Reidemeister numbers. This subgroup, in turn, admits $\Aff(R)$ (resp. $\Aff^+(R)$) as a characteristic quotient, allowing us to pass over to their diagonal part (which is also a characteristic quotient). This strategy is summarized pictorially in the following diagram for $\phee$ an automorphism of $G(R) = \PB_n(R)$.

\begin{center}
\begin{tikzpicture}[scale=0.9, every node/.style={scale=0.9}]
 \node (A4) {$\underset{\circlearrowright \phee}{\PB_n(R)} = \left[\begin{smallmatrix} \ast & \ast & \ast & \ast & \ast \\ & \ast & \ddots & \ast & \ast \\ & & \ddots & \ddots & \ast \\& & & \ast & \ast \\  & & & & \ast \end{smallmatrix}\right] \leq \mathbb{P}\mathrm{GL}_n(R)$}; 
 \node (nada2) [below=of A4] {};
 \node (nada3) [below=of nada2] {}; 
	 \node (H4) [left=of nada3] {$\overset{\phee_1 \curvearrowright}{\mb{W}_n(R)}:=\left[\begin{smallmatrix} \textcolor{red}{\ast} & 0 & 0 & 0 & \textcolor{red}{\ast} \\ & \ast & \ddots & 0 & 0 \\ & & \ddots & \ddots & 0 \\& & & \ast & 0 \\  & & & & \textcolor{red}{\ast} \end{smallmatrix}\right]$};
    \draw[arrows = {Hooks[right]->}] (H4) to node [auto] {$\swarrow$} (A4);
		 	  \node (H3) [right=of nada2] {$\overset{\psi \curvearrowright}{\mathbb{A}\mathrm{ff}(R) = \begin{pmatrix} \textcolor{red}{\ast} & \textcolor{red}{\ast} \\ 0 & \textcolor{red}{1} \end{pmatrix}}$};
				\node (H5) [right=of H3] {$\overset{\psi_{\mathrm{tf}} \curvearrowright}{\begin{pmatrix} \textcolor{blue}{\ast} & 0 \\ 0 & 1 \end{pmatrix}}$};
   \draw[arrows = {->>}] (H4) to (H3);
   \draw[arrows = {->>}] (H3) to (H5);
 %}
 \end{tikzpicture}
 \end{center}

Assume henceforth that~$R$ is an integral domain. 
By \cref{lem:desempre} it suffices to consider the case of $G(R) = \PB_n(R)$ (resp. of $G^{+}(R) = \PBplus_n(R)$) since $\PB_n(R)$ is a characteristic quotient of $\mbB_n(R)$ (resp. $\PBplus_n(R)$ is a characteristic quotient of $B_n^+(R)$). Let us fix throughout $\phee \in \Aut(\PB_n(R))$ (resp. $\phee \in \Aut(\PBplus_n(R))$) an arbitrary automorphism. 

We break the proof down in several small steps, and start by defining the relevant subgroup. Recall that $Z(\mbU_n(R))$ is the center of $\mbU_n(R)$, which is invariant under the (conjugation) action of the diagonal part $\frac{\mbD_n(R)}{Z_n(R)}$ (resp. $\frac{D_n(R_{\mathrm{tf}})}{Z_n(R_{\mathrm{tf}})}$). We can thus set 
 \begin{align} \label{def:Wn}
 \begin{split}
 \mbW_n(R):= Z(\mbU_n(R)) \rtimes \frac{\mbD_n(R)}{Z_n(R)}\leq \PB_n(R), \\
 W_n(R) := Z(\mbU_n(R)) \rtimes \frac{D_n(R_{\mathrm{tf}})}{Z_n(R_{\mathrm{tf}})} \leq \PBplus_n(R).
 \end{split}
 \end{align}
We remark that, in the case $n = 2$, one has 
\[Z(\mbU_2(R)) = \mbU_2(R) = \mc{E}_{1,2}(R) :=\langle \{ \ekl{1,2}(r) \in \mbU_n(R) \mid r \in R \}\rangle,\] whence
\begin{equation} \label{eq:W2PB2}
\mbW_2(R) = \PB_2(R) \quad \text{ and } \quad W_2(R) = \PBplus_2(R).
\end{equation}

We now show that an automorphism $\phee$ of the starting group yields a well-behaved automorphism $\phee_1$ of $\mbW_n(R)$ or $W_n(R)$. We concentrate from now on on the case of $G(R)$ since the one of $G^{+}(R)$ is entirely analogous.

Recall that each element $p$ in the split extension $\PB_n(R) = \mbU_n(R) \rtimes \frac{\mbD_n(R)}{Z_n(R)}$ can be expressed uniquely in the form $p=x[d]$ with $x \in \mbU_n(R)$ and $[d] \in \mbD_n(R)/Z_n(R)$; cf. \cref{sec:osgrupos}. Using notation similar to that of the proof of \cref{pps:phee}, for any given $[d] \in \mbD_n(R)/Z_n(R)$ we denote by $n_{[d]} \in \mbU_n(R)$ and by $[d_\phee] \in \mbD_n(R)/Z_n(R)$ the (unique) elements satisfying $\phee([d])=n_{[d]}[d_\phee]$. 
Since $\mbU_n(R)$ is characteristic in $\PB_n(R)$ by \cref{pps:char}, the map $\varphi$ induces an automorphism $\barra{\phee}$ 
on the diagonal part $\PB_n(R)/\mbU_n(R) \cong \mbD_n(R)/Z_n(R)$. 
More explicitly, the map $\barra{\phee}$ induced by $\phee$ on $\PB_n(R)/\mbU_n(R) \cong \mbD_n(R) / Z_n(R)$ is given by 
\[ \barra{\phee}: \quad [d] \mapsto [d_\phee],\] 
where $[d_\phee]$ is as above. 
Now define $\phee_1: \mbW_n(R) \to \mbW_n(R)$ by setting 
\begin{enumerate}
	\item $\phee_1(z)=\phee(z)$ for all $z \in Z(\mbU_n(R))$, 
	\item $\phee_1([q])=[q_\phee]$ for all $[q] \in \frac{\mbD_n(R)}{Z_n(R)}$.
\end{enumerate} 

\begin{lem} \label{lem:automorfismovalido}
The map $\phee_1 : \mb{W}_n(R) \to \mb{W}_n(R)$ is an automorphism.
\end{lem}

\begin{proof}
We note that $\mb{W}_n(R)$ is a {split} extension of the {abelian} $\phee$-invariant normal subgroup $Z(\mb{U}_n(R))$ by the abelian group $\mb{D}_n(R)/Z_n(R)$. Thus the fact that $\phee_1$ is an automorphism is proved exactly as we have shown that the map $\phee_0$ from the proof of \cref{pps:phee} is an automorphism. 
\end{proof}

Our next goal is to construct the automorphism $\psi \in \Aut(\Aff(R))$ that relates to $\phee$. To this end we check that $\Aff(R)$ is a characteristic quotient of $\mb{W}_n(R)$. 

\begin{lem}\label{centerW} The center $Z(\mbW_n(R))$ is given by
\[Z(\mbW_n(R))=\left\{\left[d_1(u_1)d_2(u_2)\dots d_n(u_n)\right] \mid u_1=u_n\right\}\leq \frac{\mbD_n(R)}{Z_n(R)}.\]
\end{lem}
\begin{proof} 

Let $w\in \mbW_n(R)$ be central. 
By definition of $\mbW_n(R)$, there are $z=e_{1,n}(r) \in Z(\mbU_n(R))$ and $q=d_1(u_1) \dots d_n(u_n) \in \mbD_n(R)$ such that $w=z[q]$.
Any element $t \in \mbW_n(R)$ is of the form $t=e_{1,n}(s)[d]$, where $s \in R$ and $d=d_1(v_1)\dots d_n(v_n)\in \mbD_n(R)$. Thus, using relations~\eqref{rel:commutators} and~\eqref{rel:projectiveconjugation}, 
\begin{align*}twt^{-1} 	& = e_{1,n}(s)[d]e_{1,n}(r)[q][d]^{-1}e_{1,n}(-s) \\
												& = e_{1,n}(s)([d]e_{1,n}(r)[d]^{-1})([q]e_{1,n}(-s)[q]^{-1})[q] \\
												& = e_{1,n}((1-u_1u_{n}^{-1})s+v_1v_{n}^{-1}r)[q].
\end{align*} 
Since $w$ is central, we must have 
\begin{equation}\label{twt}
r=(1-u_1u_{n}^{-1})s+v_1v_{n}^{-1}r, 
\end{equation}
for all $s\in R$ and all $v_1,v_2 \in R^\times$. In particular, for $v_1\neq v_n$ and  $s= 0$, we obtain 
\[r= v_1v_{n}^{-1}r,\] which is only possible if $r=0$. Thus, \cref{twt} reduces to
\[(1-u_1u_{n}^{-1})s=0.\] For $s\neq 0$, this is only possible if $u_1=u_n$. 
\end{proof}

\begin{lem} \label{lem:isoAff} 
The map $f : \mb{W}_n(R) \to \Aff(R)$ given by 
\[f(e_{1,n}(r)[d_1(u_1)\dots d_n(u_n)])=e_{1,2}(r)d_1(u_1u_{n}^{-1})\]
induces an isomorphism $\barra{f} : \tfrac{\mbW_n(R)}{Z(\mbW_n(R))} \to \Aff(R)$.
\end{lem}

\begin{proof} 
We want to show that $f$ is an epimorphism satisfying $\ker(f)=Z(\mbW_n(R))$. One can easily verify that $f$ is a well-defined group homomorphism. To see that $f$ is onto, notice that any given element $a=e_{1,2}(s)d_1(u)\in \Aff(R)$ with $s\in R$ and $d_1(u) \in \mbD_2(R)$ is the image of $e_{1,n}(s)[d_1(u)] \in \mbW_n(R)$, where $d_1(u) \in\mbD_n(R)$.

Now, an element $e_{1,n}(r)[d_1(u_1)\dots d_n(u_n)] \in \mbW_n(R) $ belongs to $\ker(f)$ if and only if $e_{1,2}(r)d_1(u_1u_{n}^{-1})$ is the $2 \times 2$-identity matrix. That is, $r=0$ and $u_1=u_n$. Thus, $\ker(f)=Z(\mbW_n(R))$ by \cref{centerW}.
\end{proof}

We note that \cref{lem:isoAff} and \cref{eq:W2PB2} also imply \cref{lem:AffisPB}. 

We can finally construct the isomorphism $\psi_{\mathrm{tf}}$ from the statement of \cref{thm:Aff}. Indeed, as $\phee_1 \in \Aut(\mb{W}_n(R))$, it induces an automorphism 
\begin{equation}
\label{eq:psi} \psi: \frac{\mbW_n(R) }{Z(\mbW_n(R))} \longrightarrow \frac{\mbW_n(R) }{Z(\mbW_n(R))},
\end{equation}
which can be interpreted as an automorphism of $\Aff(R)$ thanks to \cref{lem:isoAff}. Note that the homomorphism $f$ from \cref{lem:isoAff} maps $\mathcal{E}_{1,n}(R) \leq \mbW_n(R)$ isomorphically onto $\mbU_2(R) = \mc{E}_{1,2}(R) \leq\Aff(R)$. Recalling that 
\[Z(\mbU_n(R)) = \mc{E}_{1,n}(R):=\langle \{ \ekl{1,n}(r) \in \mbU_n(R) \mid r \in R \}\rangle,\]
the identification $\mbW_n(R)/Z(\mbW_n(R)) \cong \Aff(R)$ given by $\barra{f}$ yields  
\[\frac{\mbW_n(R)/Z(\mbW_n(R))}{Z(\mbU_n(R))} \cong \frac{\Aff(R)}{\mbU_2(R)}.\] 
Altogether, the map $\psi$ induces an isomorphism $\barra{\psi}$ on the diagonal part 
\[\frac{\mbW_n(R)/Z(\mbW_n(R))}{Z(\mbU_n(R))} \cong \frac{\Aff(R)}{\mbU_2(R)} \cong \mc{D}_1(R^\times) = \langle d_1(u) \mid u \in R^{\times}\rangle \cong R^\times.\]
As the torsion subgroup $t(R^\times) \leq R^\times$ is also characteristic, we may define $\psi_{\mathrm{tf}}$ as the automormphism induced by $\barra{\psi}$ on $R^\times_{\mathrm{tf}}$. (Notice that, in the `plus' case --- that is, starting with $\PBplus_n(R)$ --- the quotient $\frac{{W}_n(R)}{Z({W}_n(R))}$ is isomorphic to $\Aff^+(R)$, so that $\frac{\Aff^+(R)}{\mbU_2(R)} \cong \mc{D}_1(R^\times_{\mathrm{tf}}) \cong R^\times_{\mathrm{tf}}$, whence $\barra{\psi} = \psi_{\mathrm{tf}}$.) By \cref{lem:desempre}, one has 
\begin{equation} \label{eq:1stinequality}
R(\phee_1) \geq R(\psi) \geq R(\barra{\psi}) \geq R(\psi_{\mathrm{tf}}).
\end{equation}

\begin{proofof}{Theorem~\ref{thm:Aff}} 
It remains to establish the inequality $R(\phee) \geq R(\psi_{\mathrm{tf}})$. To this end, we consider maps $\barra{\phee}$ and $\barra{\barra{\phee}}$ induced by $\phee$ on other quotients of $\PB_n(R)$. 

By \cref{pps:char} the original automorphism $\phee \in \Aut(\PB_n(R))$ induces an automorphism $\overline{\phee}$ on 
\[\frac{\PB_n(R)}{\mbU_n(R)} \cong \frac{\mbD_n(R)}{Z_n(R)}.\]
From \cref{centerW} we know that the center $Z(\mbW_n(R))$ of $\mbW_n(R)$ is actually contained in the diagonal part $\tfrac{\mbD_n(R)}{Z_n(R)}$ of $\PB_n(R)$. While $Z(\mbW_n(R))$ is not necessarily a characteristic subgroup of $\PB_n(R)$, we claim that $Z(\mbW_n(R))$ is $\barra{\phee}$-invariant. 

To see this, recall that $\mbU_n(R)$ is characteristic in $\PB_n(R)$, whence so is its center $Z(\mbU_n(R)) = \mc{E}_{1,n}(R)$. Hence $\phee \in \Aut(\PB_n(R))$ maps $e_{1,n}(1)$ to some $e_{1,n}(r)$ where $r\neq 0$. But by definition any element $[q] \in Z(\mbW_n(R))$ centralizes the whole subgroup $\mc{E}_{1,n}(R) \leq \mbW_n(R)$. Thus, writing $\phee([q]) = n_{[q]} [q_\phee]$ uniquely with $n_{[q]} \in \mbU_n(R)$ and $[q_\phee] \in \mbD_n(R)/Z_n(R)$ and picking a representative $[d_1(a_1) \cdots d_n(a_n)] = [q_\phee]$, it follows from relation~\eqref{rel:projectiveconjugation} that 
\begin{align*}
e_{1,n}(r) 	& = \phee([q]e_{1,n}(1)[q]^{-1}) \\
			& = n_{[q]}[q_\phee]e_{1,n}(r)[q_\phee]^{-1}n_{[q]}^{-1} \\
			& = e_{1,n}(a_1 a_{n}^{-1}r),
\end{align*}
which implies $a_1 = a_n$ (because $R$ is an integral domain and $r\neq0$) and thus $[q_\phee]\in Z(\mbW_n(R))$. Hence $Z(\mbW_n(R))$ is indeed $\barra{\phee}$-invariant as a subgroup of $\PB_n(R)/\mbU_n(R) \cong \mbD_n(R) / Z_n(R)$.

We can therefore consider the automorphism $\barra{\barra{\phee}}$ induced by $\barra{\phee} \in \Aut(\tfrac{\PB_n}{\mbU_n(R)})$ on the quotient $\frac{\PB_n(R)/\mbU_n(R)}{Z(\mbW_n(R))}$. Again by \cref{lem:desempre}, we have 
\begin{equation}\label{eq:2ndinequality}
R(\phee) \geq R(\barra{\phee}) \geq R(\barra{\barra{\phee}}).
\end{equation} 
We claim that the automorphism $\overline{\psi}$, when regarded as an automorphism of $\tfrac{\mbD_n(R)/Z_n(R)}{Z(\mbW_n(R))}$, coincides with~$\barra{\barra{\phee}}$. Combined with inequalities~\eqref{eq:1stinequality} and~\eqref{eq:2ndinequality} this yields 
\[R(\phee) \geq R(\barra{\barra{\phee}}) = R(\barra{\psi}) \geq R(\psi_{\mathrm{tf}}),\] 
which will conclude the proof of the theorem. To do so, we argue that 
\begin{equation}\label{mainiso}
\frac{\mbW_n(R)/Z(\mbW_n(R))}{Z(\mbU_n(R))} \cong \frac{\PB_n(R)/\mbU_n(R)}{Z(\mbW_n(R))},
\end{equation}
and that $\barra{\psi}=\barra{\barra{\phee}}$ under this above identification.

Recall that $\mc{D}_k(R^\times) := \langle d_k(u) \mid u \in R^{\times}\rangle$ and that $\mbW_n(R) = \mathcal{E}_{1,n}(R) \rtimes \frac{\mbD_n(R)}{Z_n(R)}$. By \cref{centerW}, 
\[Z(\mbW_n(R)) \cong \frac{Z_n(R) \prod_{i=2}^{n-1}\mathcal{D}_i(R^\times)}{Z_n(R)}.\]
Thus  
\[\frac{\mbW_n(R)}{Z(\mbW_n(R))}\cong \mathcal{E}_{1,n}(R)\rtimes \frac{\mbD_n(R)}{Z_n(R) \prod_{i=2}^{n-1}\mathcal{D}_i(R^\times)}\]
because $Z_n(R) \prod_{i=2}^{n-1}\mathcal{D}_i(R)$ acts trivially on $Z(\mbU_n(R)) = \mathcal{E}_{1,n}(R)$. Moreover,
\[\frac{\mbW_n(R)/Z(\mbW_n(R))}{\mathcal{E}_{1,n}(R)}\cong \frac{\mbD_n(R)}{Z_n(R)\prod_{i=2}^{n-1}\mathcal{D}_i(R^\times)} \cong \frac{\mbD_n(R)/Z_n(R)}{Z(\mbW_n(R))}.\] 
In fact, under the identification $\tfrac{\mbW_n(R)/Z(\mbW_n(R))}{\mathcal{E}_{1,n}(R)} \cong \tfrac{\mbD_n(R)/Z_n(R)}{Z(\mbW_n(R))}$, an element 
\[\mathcal{E}_{1,n}(R)(Z(\mbW_n(R))\cdot [d]) \in \frac{\mbW_n(R)/Z(\mbW_n(R))}{\mathcal{E}_{1,n}(R)}\] corresponds to the element 
\[Z(\mbW_n(R)) \cdot [d] \in \frac{\mbD_n(R)/Z_n(R)}{Z(\mbW_n(R))}.\] 
Therefore $\overline{\psi}$, viewed as an automorphism of the latter quotient group, is given by 
\[Z(\mbW_n(R))\cdot [d] \mapsto Z(\mbW_n(R))\cdot [d_{\phee}],\] 
which is exactly~$\barra{\barra{\phee}}$, proving our claim and hence the theorem. 
\end{proofof}

\section{Applications of Theorem~\ref{thm:Aff} to arithmetic groups}\label{Applications} 

We now provide a series of infinite families of finitely presented soluble $S$-arithmetic groups having property~$\Ri$. (Recall that a group $\Gamma$ is called an $S$-arithmetic subgroup of $\mbG(\K)$ if there exists a Dedekind domain of arithmetic type $\OS$ in a global field $\K$ and a linear algebraic group $\mbG$ defined over $\K$ such that $\Gamma$ is commensurable with $\rho^{-1}(\GL_n(\OS)) \leq \mbG(\K)$ for some faithful $\K$-representation $\rho$ of $\mbG$ into $\GL_n$ \cite{Margulis}.)

Throughout this section we work with the following Dedekind domains of arithmetic type: 
\begin{itemize}
\item $R=\Z\left[\tfrac{1}{{m}}\right] \subset \Q$, where ${m} \in \N_{\geq 2}$; 
\item $S=\F_q[t,t^{-1},f(t)^{-1}] \subset \F_q(t)$, where $q$ is a power of a prime $p$ and $f(t) \in \F_p[t] \setminus \set{t} \subset \F_q[t]$ is a nonconstant monic polynomial which is irreducible over $\F_q[t]$. In addition, if $p=2$, we require $f(t) \neq t-1$; 
\item $\ri = \ri_{\Q(\sqrt{d})}$, the ring of integers of a quadratic number field $\Q(\sqrt{d})$, where $d \in \N_{\geq 2}$ is a square-free positive integer for which a fundamental unit $\eta \in \ri_{\Q(\sqrt{d})}^\times$ has minimal polynomial of the form $x^2 +ax -1 \in \Z[x]$. 
\end{itemize} 
We prove the following.
 
\begin{thm}\label{pps:zft} Let $R=\Z\left[\tfrac{1}{m}\right]$, $S=\F_q[t,t^{-1},f(t)^{-1}]$, and $\ri = \ri_{\Q(\sqrt{d})}$ be as above. Then the following hold. 
\begin{enumerate}
	\item \label{pps:zft.i} 
	The automorphism induced by any $\psi \in \Aut(A(R))$, where $A \in \{\Aff,\Aff^+\}$, on the diagonal part $\left(\begin{smallmatrix} * & 0 \\ 0 & 1 \end{smallmatrix} \right)$ fixes infinitely many points. 
	\item \label{pps:zft.ii} 
	The automorphism induced by any $\psi \in \Aut(\Aff^+(S))$ on the diagonal part $\left(\begin{smallmatrix} * & 0 \\ 0 & 1 \end{smallmatrix} \right)$ fixes infinitely many points. 
	\item \label{pps:zft.iii} 
	The automorphism induced by any $\psi \in \Aut^+(\Aff(\ri))$ on the diagonal part $\left(\begin{smallmatrix} * & 0 \\ 0 & 1 \end{smallmatrix} \right)$ fixes infinitely many points. 
\end{enumerate}
In particular, for all $n \geq 2$, the groups $\mbB_n(R)$, $\PB_n(R)$, $\Aff(R)$, $B_{n}^{+}(R)$, $\PBplus_{n}(R)$, $\Aff^{+}(R)$, $B_{n}^{+}(S)$, $\PBplus_{n}(S)$, $\Aff^{+}(S)$, $B_{n}^{+}(\ri)$, $\PBplus_{n}(\ri)$, and $\Aff^{+}(\ri)$ have property~$\Ri$.
\end{thm}

The fact that all groups in the above list are finitely presented follows from known results about $S$-arithmetic groups; see \cite{Abels, Bux0, YuriSoluble}. 

We point out that \cref{pps:zft} recovers some results that were previously established in the literature using different methods, but also extending them to new cases. For instance, the fact that $\mbB_n(R)$, $\PB_n(R)$, $\Aff(R)$, $B_{n}^{+}(R)$, $\PBplus_{n}(R)$ and $\Aff^{+}(R)$ have~$\Ri$ ($R = \Z[\frac{1}{m}]$) also follows from a result of Gon\c{c}alves--Kochloukova \cite[Theorem~4.3]{DesiDaciberg} combined with the fact that such $S$-arithmetic groups are nonpolycyclic and, by a result of Tiemeyer~\cite{Tiemeyer}, of homotopical type $\mathtt{F}_\infty$. Methods due to Gon\c{c}alves--Wong~\cite{DacibergWongCounterexample} and Dekimpe--Tertooy--Van den Bussche~\cite{KarelSamIrisSolv} can also be used to check whether the (torsion-free metabelian) groups $\Aff^+(\ri_{\Q(\sqrt{d})})$ have~$\Ri$. Additionally, our alternative criterion to test for~$\Ri$ in $\mbB_n(\ri_{\Q(\sqrt{d})})$ and $\PB_n(\ri_{\Q(\sqrt{d})})$ given in the companion paper~\cite{Bn2} shows that these groups have property~$\Ri$ for $n\geq 4$.

At the same time, \cref{pps:zft} complements or improves on the above mentioned earlier results. For instance, Gon\c{c}alves--Kochloukova~\cite{DesiDaciberg} also showed that soluble groups $G$ of homological type $\mathtt{FP}_{2}$ for which $G/G''$ is not polycyclic `almost have~$\Ri$' in a precise sense \cite[Corollary~4.7]{DesiDaciberg}. While their corollary applies to $B_{n}^{+}(S)$, $\PBplus_{n}(S)$ and $\Aff^{+}(S)$ (with $S = \F_q[t,t^{-1},f(t)^{-1}]$), it does not guarantee that these groups have~$\Ri$. To the best of our knowledge, ours is the first proof of this fact. These examples are of interest also for their metric and self-similarity properties \cite{BartholdiNeuhauserWoessDL, DesiSaidSSFPn}. 

Furthermore, the tools of Gon\c{c}alves--Kochloukova~\cite{DesiDaciberg} do not apply to polycyclic groups. Thus the examples $B_{n}^{+}(\ri_{\Q(\sqrt{d})})$, $\PBplus_{n}(\ri_{\Q(\sqrt{d})})$, $\Aff^{+}(\ri_{\Q(\sqrt{d})})$ (with $n\geq 2$) from \cref{pps:zft} and $\mbB_n(\ri_{\Q(\sqrt{d})})$, $\PB_n(\ri_{\Q(\sqrt{d})})$ (with $n \geq 4$) from the companion article~\cite{Bn2} form the first infinite series of nonnilpotent polycyclic groups of increasing solubility class and with property~$\Ri$. Part~\eqref{pps:zft.iii} of \cref{pps:zft} generalizes our earlier \cref{pps:goldenRi} with the golden ratio. For example, $\ri_{\Q(\sqrt{2})}$ fits into that framework, so that $B_{n}^{+}(\Z[\sqrt{2}])$, $\PBplus_{n}(\Z[\sqrt{2}])$ and $\Aff^{+}(\Z[\sqrt{2}])$ all have~$\Ri$.

Our findings and the current state of knowledge discussed above prompt us to state the following. 

\begin{conj} \label{theconjecture}
Let $\mbB$ be a Borel subgroup of a split, connected, reductive, noncommutative linear algebraic group $\mbG$, all defined over a global field $\K$. If $\Gamma$ is an $S$-arithmetic subgroup of $\mbB(\K)$ that does not have property~$\Ri$, then $\Gamma$ is either virtually polycyclic or not finitely presented.
\end{conj}

The conclusion that the groups listed in \cref{pps:zft} have property~$\Ri$ is an immediate application of \cref{thm:Aff}. We thus restrict ourselves to showing that any automorphism of the given affine groups induces the identity on the diagonal. 
Parts~\eqref{pps:zft.i}, \eqref{pps:zft.ii} and~\eqref{pps:zft.iii} of \cref{pps:zft} are proved in Sections~\ref{zom}, \ref{fqtmoft}, and~\ref{ririri}, respectively. 

\subsection{Proof of Theorem~\ref{pps:zft}\eqref{pps:zft.i}}\label{zom} 
Let $p_1, \dots, p_{{k}}$ be the distinct prime integers in the decomposition of $m \in \N_{\geq 2}$ and  write $R=\Z\left[\tfrac{1}{p_1 \cdots p_k}\right]$. We restrict ourselves to the case of $\Aff(R)$ --- it will be clear from the arguments that the case of $\Aff^+(R)$ is analogous. 

By \cref{pps:phee} we may assume that $\psi(\mbU_2(R))\subseteq \mbU_2(R)$ and $\psi(\mathcal{D}_1(R)) \subseteq \mathcal{D}_1(R)$. 
Denote by $\psi_{\mathrm{tf}}$ the automorphism induced by %$\barra{\psi}$ 
$\psi$ on the torsion-free part of $\sbgpdk{1}(R^\times) \cong \Aff(R)/\mbU_2(R)$, that is, the automorphism induced on
\[\frac{\mathcal{D}_1(R^\times)}{t(\sbgpdk{1}(R^\times))} \cong \mathcal{D}_1(R^\times_{\mathrm{tf}}) = 
\gera{ \set{ d_1(p_{i}) \mid i =1, \ldots, k} }\cong \Z^k.\] 
Since $\mbU_2(R)$ is characteristic in $\Aff(R)$, there exists $r\in R\setminus \{0\}$ such that 
\[\psi(e_{1,2}(1))=e_{1,2}(r).\]
Moreover, as $\psi(\mathcal{D}_1(R)) \subseteq \mathcal{D}_1(R)$, there are for each $j \in \{1, \dots, k\}$ some $\ell_j, \lambda_{1j}, \dots, \lambda_{kj}\in \Z$ such that 
\[\psi(d_1(p_j))=d_1(-1)^{\ell_j} d_1(p_1)^{\lambda_{1j}}\cdots d_1(p_k)^{\lambda_{kj}}.\] 
Using the equality
\[e_{1,2}(p_j)=d_1(p_j)e_{1,2}(1)d_1(p_j)^{-1},\]
we see that 
\begin{align*}e_{1,2}(rp_j) &=\psi(e_{1,2}(p_j)) 
							 =\psi\left(\left(d_1(p_j)e_{1,2}(1)d_1(p_j)^{-1}\right)\right)\\
							&=e_{1,2}((-1)^{\ell_j} p_{1}^{\lambda_{1j}}\cdots p_{k}^{\lambda_{kj}}r).
\end{align*}
Consequently, 
\[\lambda_{ij}=	\begin{cases} 	1, &\text{ if }i=j,\\
																0, &\text{ otherwise}.
								\end{cases}\]
It follows that $\psi_{\mathrm{tf}}$ (as a an element of $\GL_k(\Z)$) is the identity, thus fixes infinitely many points. Hence $\mbB_n(R)$, $\PB_n(R)$ and $\Aff(R)$ have $\Ri$ by \cref{thm:Aff}. The previous computations also cover the case of $\Aff^+(R)$. \qed

\subsection{Proof of Theorem~\ref{pps:zft}\eqref{pps:zft.ii}}\label{fqtmoft}
Let $q$ be a power of a prime $p$ and let $f(t)$ be a nonconstant monic polynomial 
\begin{equation}\label{ftdef}f(t)= \sum_{k=0}^{m}\lambda_kt^k \in \F_p[t]\setminus\{t\} \subset \F_q[t] \end{equation}
that is irreducible over $\F_q[t]$. In particular $\lambda_0 \neq 0$ and $\lambda_m = 1$. If $p=2$, we require additionally that $f(t)\neq t-1$. Throughout this section, we write $S=\F_q[t,t^{-1},f(t)^{-1}]$. 

Here, we take $S^\times_{\mathrm{tf}}=\langle t^k, f(t)^\ell \mid k,\ell \in \Z\rangle \cong \Z^2$ as a complement for the torsion subgroup of units $t(S^\times) = \F_q \setminus \{0\}$, so that 
the group $\Aff^+(S)$ is given by  
\[
\Aff^+(S) = \left\{\left( \begin {array}{cc} t^kf(t)^\ell &s\\ \noalign{\medskip}0&1\end {array} \right) \, \middle| \, k, \ell \in \Z,~ s \in S \right\}.
\] 
Let $\psi \in \Aut(\Aff^{+}(S))$. We argue that the induced automorphism $\psi_{\mathrm{tf}}:=\overline{\psi}$ on $\mc{D}_1(S^\times_{\mathrm{tf}}) \cong \Aff^{+}(S)/\mbU_2(S) \cong \Z^2$ has eigenvalue $1$ (interpreted as an element of $\GL_2(\Z)$) 
--- equivalently by \cref{lem:ReidemeisterAbelian}, $\psi_{\mathrm{tf}}$ has infinitely many fixed points.

Again by \cref{pps:phee} we may assume that $\psi(\mathcal{D}_1(S^\times_{\mathrm{tf}})) \subseteq \mathcal{D}_1(S^\times_{\mathrm{tf}})$. 
Thus, there are $a,b,c,d \in \Z$ such that 
\begin{align*} 	\psi(d_1(t))	&=d_1(t^af(t)^b),\\
							\psi(d_1(f(t)))	&=d_1(t^cf(t)^d).
\end{align*}
In particular, for $k\in \Z$, we have
\[\psi(d_1(t^k))=\psi(d_1(t))^k=d_1(t^{ak}f(t)^{bk}).\]

Observe that, when regarded as an automorphism of $\Z^2\cong \mathcal{D}_1(S^\times_{\mathrm{tf}})$, the map $\psi_{\mathrm{tf}}$ is represented by the matrix 
\[M_{\psi_{\mathrm{tf}}}=\begin{pmatrix} a & c \\ b & d \end{pmatrix} \leq \GL_2(\Z). \] 
In particular, its determinant $\det(M_{\psi_{\mathrm{tf}}})=ad-bc$ must be $\pm 1$.
To show that $R(\psi_{\mathrm{tf}})=\infty$, it suffices to prove that 
$\det(\mb{1}_2-M_{\psi_{\mathrm{tf}}})=0$.

The subgroup $\mbU_2(S)$ is characteristic in $\Aff^{+}(S)$, so that there exists an automorphism $\Phi: \Addi(S) \to \Addi(S)$ satisfying 
\[\psi(e_{1,2}(s))=e_{1,2}(\Phi(s)), \text{ for each }s\in S.\] 
For each $u\in S^\times_{\mathrm{tf}}$ and each $s \in S$, the following holds in $\Aff^{+}(S)$ by relations~\eqref{rel:commutators}:
\[e_{1,2}(us)=d_1(u)e_{1,2}(s)d_1(u)^{-1}.\]
Thus, for each $k \in \N$, 
\[e_{1,2}(\Phi(t^k))= \psi(e_{1,2}(t^k))=\psi(d_1(t^k)e_{1,2}(1)d_1(t^k)^{-1})=e_{1,2}(t^{ak}f(t)^{bk}\Phi(1)),\]
which implies $\Phi(t^k)=t^{ak}f(t)^{bk} \Phi(1)$. 
On the one hand, $\Phi$ is an automorphism of $\Addi(S)$, so that 
\[\Phi(f(t))=\sum_{k=0}^{m}\lambda_k\Phi(t^k)=\sum_{k=0}^{m}\lambda_kt^{ak}f(t)^{bk}\Phi(1)\]
because $\lambda_k \in \F_p$ for all $k\in \{0, 1,\dots, m\}$. 
On the other hand, 
\begin{align*}e_{1,2}(\Phi(f(t))	&=\psi(e_{1,2}(f(t))=\psi(d_1(f(t))e_{1,2}(1)d_1(f(t))^{-1})\\
																	&=e_{1,2}(t^cf(t)^d\Phi(1)).
\end{align*}
Whence 
\begin{equation}
\label{eq:cdab} t^cf(t)^d=\sum_{k=0}^{m}\lambda_kt^{ak}f(t)^{bk}.
\end{equation}

Let us determine which values of $a,b,c,d \in \Z$ are solutions of \cref{eq:cdab} and, for each one of them, we show that the corresponding matrix $M_{\psi_{\mathrm{tf}}}$ has eigenvalue~$1$. 

We use $t$-adic and $f(t)$-adic valuations and a case-by-case analysis; refer, e.g., to \cite{Rosen} for background on function fields. Let us recall the definitions of these valuations. Consider the field $\F_q(t)$ of rational functions on the variable~$t$. Since $\F_q[t]$ is a principal ideal domain and $f(t)$ is irreducible in $\F_q[t]$, 
each $x \in \F_q(t)$ can be written 
uniquely {(up to reordering of factors)}  as 
\[x=f(t)^{\nu_x} \pi_{1}^{a_1}\dots \pi_{j}^{a_j},\]
where $\nu_x, a_1, \dots, a_j$ are integers and the $\pi_i$ are irreducible elements of $\F_q[t]$. 
The $f(t)$-adic valuation $v_{f(t)}: \F_q(t) \to \Z \cup \{\infty\}$ is defined by the rule $v_{f(t)}(x)= \nu_x$ and $v_{f(t)}(0)=\infty$. In particular, this restricts to the (nonnegative) $f(t)$-adic valuation on $\F_q[t]$, also written $v_{f(t)}: \F_q[t] \to \Z_{\geq 0} \cup \{\infty\}$ by abuse of notation. 
Analogously, we define a $t$-adic valuation $v_t: \F_q(t) \to \Z\cup \{\infty\}$. 

\textbf{\underline{Case 1:}} $\mathbf{a>0.}$
We split Case~1 according to whether $b$ is positive, negative or zero.

\textbf{\underline{Case 1.1:}} $\mathbf{a>0}$ \textbf{and} $\mathbf{b<0.}$
Write $b=-\beta$ with $\beta$ a positive integer. Since we want to compare $f(t)$ and $t$-adic valuations, we need all powers in \cref{eq:cdab} to be nonnegative. We then rewrite \cref{eq:cdab} as
\[ t^cf(t)^{d+\beta m}=\sum_{k=0}^{m}\lambda_kt^{ak}f(t)^{\beta(m-k)}.\]
Notice that the RHS equals
\[\lambda_0f(t)^{\beta m}+ t^a\sum_{k=1}^{m}\lambda_kt^{a(k-1)}f(t)^{\beta(m-k)},\]
which has $t$-adic valuation zero (i.e. it is not divisible by $t$), since $\lambda_0\neq 0$ and $\beta m>0$. Consequently, $c=0$. 
This gives \[\pm 1=ad-bc =ad ~ \text{ so that } ~ a=1, d=\pm 1.\]
For these values of $a$ and $c$, it is clear that $\det(\mb{1}_2-M_{\psi_{\mathrm{tf}}})=0$. 
\smallskip 

\textbf{\underline{Case 1.2:}} $\mathbf{a>0}$ \textbf{and} $\mathbf{b>0.}$
In this case, we can rewrite \cref{eq:cdab} as follows: 
\[ t^cf(t)^d= \lambda_0+t^a \sum_{k=1}^{m}\lambda_kt^{a(k-1)}f(t)^{bk}.\]
Since $\lambda_0 \neq 0$ and $a>0$, we see that the $t$-adic valuation of the RHS is zero (i.e., the RHS is not divisible by $t$). Thus, $c=0$.

Thus, we can write \cref{eq:cdab} as
\[ f(t)^d= \lambda_0+f(t)^b \sum_{k=1}^{m}\lambda_kt^{ak}f(t)^{b(k-1)}.\]
As above, we conclude that $d=0$. This is a contradiction to $ad-bc= \pm 1$.
Thus, we cannot have $a$ and $b$ both positive. 
\smallskip 

\textbf{\underline{Case 1.3:}} $\mathbf{a>0}$ \textbf{and} $\mathbf{b=0.}$ 
In this case, we must have $ad=\pm 1$ and hence $a=1$. 
For these values of $a$ and $b$, it is clear that $\det(\mb{1}_2-M_{\psi_{\mathrm{tf}}})=0$.  
\medskip

\textbf{\underline{Case 2:}} $\mathbf{a<0.}$
Here, $a=-\alpha$ with $\alpha$ a positive integer. 
As before, we split this case according to whether $b$ is positive, negative or zero.

\textbf{\underline{Case 2.1:}} $\mathbf{a<0}$ \textbf{and} $\mathbf{b>0.}$
In this case, \cref{eq:cdab} becomes: 
\[t^{c+\alpha m}f(t)^d=\sum_{k=0}^{m}\lambda_kt^{\alpha(m-k)}f(t)^{bk}= \lambda_0t^{\alpha m}+f(t)^b\sum_{k=1}^{m}\lambda_kt^{\alpha(m-k)}f(t)^{b(k-1)}.\]
The $f(t)$-adic valuation of the RHS is zero (i.e., the RHS is not divisible by $f(t)$), thus $d=0$. We then conclude 
\[t^{c+\alpha m}=f(t)^{bm}+t^{\alpha}\sum_{k=0}^{m-1}\lambda_kt^{\alpha(m-k-1)}f(t)^{bk}.\]
Again, since the RHS has $t$-adic valuation zero, we must have $c=-\alpha m=am$.
In particular, $\pm 1=ad-bc = bc=bam$, hence $a=-1$, $b=1$, $c=-1$, $m=1$. 
Here we do obtain $\det(\mb{1}_2-M_{\psi_{\mathrm{tf}}})\neq 0$. However, substituting the values above for $a$, $b$, $c$, $d$, $m$ in expressions~\eqref{ftdef} and~\eqref{eq:cdab}, it follows that 
\[ f(t) = \lambda_0 + t \quad \text{ and } \quad f(t) = 1-\lambda_0 t. \]
Since $f(t)$ is monic, it holds $\lambda_0 = -1$ and thus 
$f(t)=t-1=t+1$, that is, $f(t)=t-1$ in characteristic two, the excluded case. Thus, we cannot have negative~$a$ and positive~$b$ simultaneously. 
\smallskip 

\textbf{\underline{Case 2.2:}} $\mathbf{a<0}$ \textbf{and} $\mathbf{b<0.}$
Suppose $b=-\beta <0$ with $\beta$ a positive integer. 
We can rewrite \cref{eq:cdab} as
\[t^{c+\alpha m}f(t)^{d+\beta m}=1+t^{\alpha}f(t)^\beta\sum_{k=0}^{m-1}\lambda_kt^{\alpha(m-k-1)}f(t)^{\beta(m-k-1)},\] so that the RHS is divisible by neither $t$ nor $f(t)$. We then must have $c=-\alpha m=am$ and $d=-\beta m=bm$. This gives 
$ad-bc=0$, a contradiction. 
Thus, we cannot have $a$ and $b$ both negative. 
\smallskip

\textbf{\underline{Case 2.3:}} $\mathbf{a<0}$ \textbf{and} $\mathbf{b=0.}$
Here, we have $ad=\pm 1$, so that $a=-1$ and $d=\pm 1$. In this case, \cref{eq:cdab} reduces to 
    \[t^{c+m}f(t)^d=\sum_{k=0}^{m}\lambda_kt^{m-k}.\]
Since the RHS is not divisible by $t$, we must have $c=-m$. Moreover, if $d=-1$, we would have 
\[1=f(t)\sum_{k=0}^{m}\lambda_kt^{m-k}\] 
as an equality in $\F_q[t]$. Since this cannot hold, one has $d=1$.  
For these values of $b$ and $d$, it is clear that  $\det(\mb{1}_2-M_{\psi_{\mathrm{tf}}})=0$.
\medskip

\textbf{\underline{Case 3:}} $\mathbf{a=0.}$ 
If $a=0$, we get $bc=\pm 1$. Let us split this case according to whether $b=-1$ or $b=1$.

\textbf{\underline{Case 3.1:}} $\mathbf{a=0}$ \textbf{and} $\mathbf{b=-1.}$
In this case, \cref{eq:cdab} reduces to
\[ t^cf(t)^{d+m}=1+\sum_{k=0}^{m-1}\lambda_kf(t)^{m-k}.\]
Again, the $f(t)$-adic valuation on the RHS is zero, so that $d=-m$. 
If $c=-1$, we would have 
\[ 1=t\sum_{k=0}^{m}\lambda_kf(t)^{m-k},\]
so it must follow $c=1$, and \cref{eq:cdab} becomes 
\[ t=1+\sum_{k=0}^{m-1}\lambda_kf(t)^{m-k}. \]
In particular, 
\begin{equation}\label{eq:cneg} t-1=f(t)\sum_{k=0}^{m-1}\lambda_kf(t)^{m-k-1},\end{equation} 
and hence $f(t)$ divides $t-1$. Since both $f(t)$ and $t-1$ are irreducible and monic, we must have $f(t)=t-1$. In particular, $m=1$ and $\lambda_0=-1$. Then, \cref{eq:cneg} reduces to
\[ t-1=\lambda_0f(t)=-t+1,\] 
so that $1=-1$, that is, $\carac(\F_q)=2$ and we get the excluded case $f(t)=t-1$. 
Thus, it is not possible to have $a=0$ and negative $b$ under our assumptions. 
\smallskip 

\textbf{\underline{Case 3.2:}} $\mathbf{a=0}$ \textbf{and} $\mathbf{b=1.}$
In this case, \cref{eq:cdab} becomes: 
\[t^cf(t)^d=\sum_{k=0}^{m}\lambda_kf(t)^{k}.\]
Since the RHS has $f(t)$-adic valuation zero, we must have $d=0$. 
If $c=-1$, then we would have 
\[1=t\sum_{k=0}^{m}\lambda_kt^{ak}f(t)^{bk},\] so that $c=1$.
It is clear that for $a=d=0$ and $b=c=1$ one has $\det(\mb{1}_2-M_{\psi_{\mathrm{tf}}})=0$. 
\qed

\subsection{Proof of Theorem~\ref{pps:zft}\eqref{pps:zft.iii}} \label{ririri}

Since $d \in \N_{\geq 2}$ is a square-free positive integer, the quadratic number field $\Q(\sqrt{d})$ is such that its ring of integers $\ri = \ri_{\Q(\sqrt{d})}$ has group of units $\ri^\times = \{\pm 1\} \cdot \gera{\eta} \cong C_2 \times \Z$, 
where $\eta$ is the fundamental unit. Now let $\psi \in \Aut(\Aff(\ri))$. As $\mbU_2(\ri)$ is characteristic in $\Aff(\ri)$, there is some $r \in \ri \setminus \{0\}$ such that 
\[\psi(\ekl{1,2}(1))=\ekl{1,2}(r).\]
By \cref{pps:phee}, we may assume that there exists $\ell \in \Z$ such that  
\[\psi(d_1(\eta)) = d_1(\eta^\ell).\]  
Note that, since $\psi$ is an automorphism and 
\[\frac{\mathcal{D}_1(\ri^\times)}{t(\sbgpdk{1}(\ri^\times))} \cong \mathcal{D}_1(\ri^\times_{\mathrm{tf}}) = 
\gera{ \set{ d_1(\eta)} } \cong \Z,\] 
we can only have $\ell=1$ or $\ell=-1$. We argue that $\ell = 1$. 

Recall that $x^2+ax-1 \in \Z[x]$ is the minimal polynomial of $\eta$. By relations~\eqref{rel:commutators}, one has 
\[ \ekl{1,2}(\eta^2) = d_1(\eta^2) \ekl{1,2}(1) d_1(\eta^{-2}), \]
whereas by~\eqref{rel:commutators} and the identity $\eta^2 + a\eta -1 = 0$ it holds 
\[\ekl{1,2}(\eta^2) = d_1(\eta)\ekl{1,2}(1)^{-a}d_1(\eta^{-1})\ekl{1,2}(1).\]
Applying $\psi$ to both equalities yields 
\[\ekl{1,2}(\eta^{2\ell}r) = \ekl{1,2}(-\eta^{\ell}ar+r).\] 
Hence $\eta^{2\ell}+a\eta^{\ell}-1 = 0$, so that $\eta^\ell$ is also a root of the (monic, irreducible) polynomial $x^2+ax-1 \in \Z[x]$. If $\ell$ were equal to $-1$, the identities $\eta^2+a\eta-1=0=\eta^{-2}+a\eta^{-1}-1$ would imply 
\[\eta^4 + a\eta^3 = 1 + a\eta,\]
so that $\eta$ would be a root of $x^4+ax^3-ax-1 \in \Z[x]$. This polynomial is divisible by $x^2-1$, whence $\eta$ would also be a root of $x^2+ax+1$. A contradiction since $x^2+ax-1$ is the minimal polynomial for $\eta$. Thus $\ell=1$ and we are done. \qed

\section*{Acknowledgments}
Part of this project was carried when PMLA was at KU Leuven Kulak, Kortrijk, supported by the long term structural funding \emph{Methusalem grant} of the Flemish Government. YSR was partially supported by the German Research Foundation (DFG) through the Priority Program 2026 \emph{Geometry at infinity}, Project~62, and the GRK~2297 \emph{MathCoRe}, 314838170. 

We are indebted to Karel Dekimpe, Daciberg Gon\c{c}alves, and Peter Wong for invaluable comments and suggestions that inspired or greatly improved this work, and particularly Dekimpe for helpful remarks on earlier versions of the present article. This project, including the companion paper~\cite{Bn2}, heavily profited from discussions at the 2018 ICM Satellite Conference ``Geometric Group Theory'' in Campinas, Brazil, the 2020 Workshop ``Geometric Structures in Group Theory'' in Oberwolfach, Germany, and the 2023 Conference ``Nielsen Theory and Related Topics'' in Oostende, Belgium. We gratefully thank the organizers of these events and the participants for their contributions.

We also thank the anonymous referee for helpful comments and suggestions.

 \def\cprime{$'$} \def\cprime{$'$}
 \providecommand{\bysame}{\leavevmode\hbox to3em{\hrulefill}\thinspace}
 \providecommand{\MR}{\relax\ifhmode\unskip\space\fi MR }

\printbibliography

\end{document}